\documentclass{amsart}
\usepackage{graphics}
\usepackage{graphicx}
\usepackage{amsfonts,amsmath}
\begin{document}

 \newtheorem{thm}{Theorem}[section]
 \newtheorem{cor}[thm]{Corollary}
 \newtheorem{lem}[thm]{Lemma}{\rm}
 \newtheorem{prop}[thm]{Proposition}

 \newtheorem{defn}[thm]{Definition}{\rm}
 \newtheorem{assumption}[thm]{Assumption}
 \newtheorem{rem}[thm]{Remark}
 \newtheorem{ex}{Example}
\numberwithin{equation}{section}

\def\la{\langle}
\def\ra{\rangle}
\def\e{{\rm e}}
\def\x{\mathbf{x}}
\def\by{\mathbf{y}}
\def\bz{\mathbf{z}}
\def\F{\mathcal{F}}
\def\R{\mathbb{R}}
\def\T{\mathbf{T}}
\def\N{\mathbb{N}}
\def\K{\mathbf{K}}
\def\bK{\overline{\mathbf{K}}}
\def\Q{\mathbf{Q}}
\def\M{\mathbf{M}}
\def\O{\mathbf{O}}
\def\C{\mathbf{C}}
\def\P{\mathbf{P}}
\def\Z{\mathbb{Z}}
\def\H{\mathcal{H}}
\def\A{\mathbf{A}}
\def\V{\mathbf{V}}
\def\AA{\overline{\mathbf{A}}}
\def\B{\mathbf{B}}
\def\c{\mathbf{C}}
\def\L{\mathbf{L}}
\def\bS{\mathbf{S}}
\def\H{\mathcal{H}}
\def\I{\mathbf{I}}
\def\Y{\mathbf{Y}}
\def\X{\mathbf{X}}
\def\G{\mathbf{G}}
\def\B{\mathbf{B}}
\def\f{\mathbf{f}}
\def\z{\mathbf{z}}
\def\y{\mathbf{y}}
\def\d{\hat{d}}
\def\bx{\mathbf{x}}
\def\y{\mathbf{y}}
\def\v{\mathbf{v}}
\def\g{\mathbf{g}}
\def\w{\mathbf{w}}
\def\b{\mathcal{B}}
\def\a{\mathbf{a}}
\def\q{\mathbf{q}}
\def\u{\mathbf{u}}
\def\s{\mathcal{S}}
\def\cc{\mathcal{C}}
\def\co{{\rm co}\,}
\def\cp{{\rm CP}}
\def\tg{\tilde{f}}
\def\tx{\tilde{\x}}
\def\supmu{{\rm supp}\,\mu}
\def\supnu{{\rm supp}\,\nu}
\def\m{\mathcal{M}}
\def\bR{\mathbf{R}}
\def\om{\mathbf{\Omega}}
\def\c{\mathbf{c}}
\def\s{\mathcal{S}}
\def\k{\mathcal{K}}
\def\la{\langle}
\def\ra{\rangle}
\def\blambda{{\boldmath{\lambda}}}
\def\bsmlambda{\boldmath{\lambda}}
\def\ov{\overline{o}}
\def\und{\underline{o}}
\title[nonnegativity]{Computing Gaussian \& exponential measures of semi-algebraic sets}

\author{Jean B. Lasserre}

\address{LAAS-CNRS and Institute of Mathematics\\
University of Toulouse\\
LAAS, 7 avenue du Colonel Roche\\
31077 Toulouse C\'edex 4,France}
\email{lasserre@laas.fr}
\date{}

\begin{abstract}
We provide a numerical scheme to approximate as closely as desired 
the Gaussian or exponential measure $\mu(\om)$ of (not necessarily compact) basic semi-algebraic sets
$\om\subset\R^n$. We obtain two  monotone (non-increasing and non-decreasing) sequences of 
upper and lower bounds $(\overline{\omega}_d)$, $(\underline{\omega}_d)$, $d\in\N$, 
each converging to $\mu(\om)$ as $d\to\infty$. For each $d$, computing $\overline{\omega}_d$ or $\underline{\omega}_d$
reduces to solving a semidefinite program whose size increases with $d$. 
Some preliminary (small dimension) computational experiments are encouraging and illustrate the
potential of the method. The method also works for any measure whose moments are known and which satisfies Carleman's condition.

\end{abstract}



\maketitle

\section{Introduction}

Given a basic semi-algebraic set
\begin{equation}
\label{setk}
\om  \,=\,\{\,\x\in\R^n:\:g_j(\x)\,\geq\,0,\:j=1,\ldots,m\,\},
\end{equation}
for some polynomials $(g_j)\subset\R[\x]$, we want to compute (or more precisely, approximate as closely as desired)
$\mu(\om)$ for the standard Gaussian measure on $\R^n$:
\begin{equation}
\label{mu-gauss}
\mu(B)\,=\,\frac{1}{(2\pi)^{n/2}}\int_B\exp(-\frac{1}{2}\Vert\x\Vert^2)\,d\x,\qquad \forall\,B\in\mathcal{B}(\R^n),
\end{equation}
and the standard exponential measure on the positive orthant $\R^n_+$:
\begin{equation}
\label{mu-exp}
\mu(B)\,=\,\frac{1}{(2\pi)^{n/2}}\int_B\exp(-\sum_{i=1}^n x_i)\,d\x,\qquad \forall\,B\in\mathcal{B}(\R^n_+).
\end{equation}
This problem is ``canonical" as for a non-centered Gaussian with density $\exp(-(\x-m)^T\Q(\x-m))$ 
(for some real symmetric positive definite matrix $\Q$)
or for an exponential measure with density $\exp(-\sum_i\lambda_i x_i)$
one may always reduce the problem to 
the above one by an appropriate change of variable. Indeed after this change of variable 
the new domain is again a basic semi-algebraic set of the form (\ref{setk}).

Computing $\mu(\om)$ has applications in Probability \& Statistics where the Gaussian measure plays a central role. 
In full generality with sets $\om$ as general as (\ref{setk}), it is a difficult and challenging problem even for rectangles $\om$, e.g.:
\begin{equation}
\label{rectangle}
\frac{1}{2\pi\sqrt{1-\rho^2}}\int_{-\infty}^{a_1}\int_{-\infty}^{a_2}\exp(-(x^2+y^2-2\rho \,xy)/2(1-\rho^2))\,dx\,dy,
\end{equation}
in small dimension like $n=2$ or $n=3$.
Indeed, citing A. Genz \cite{genz1}: {\it bivariate and trivariate probability distributions computation are needed for many statistics applications $\ldots$ high quality algorithms for bivariate and trivariate probability distribution computations have only more recently
started to become available.} For instance, Genz \cite{genz1} describes techniques with high accuracy results for bivariate and trivariate 
``rectangles" (\ref{rectangle}) using sophisticated techniques to integrate Plackett's formula. Again 
those efficient techniques takes are very specific as they take advantage of Plackett's formula available for
(\ref{rectangle}). Interestingly, a (complicated) formula in closed form is provided in Chandramouli and Ranganathan \cite{chandramouli}
via the characteristic function method.

The case of ellipsoids $\om$ has been investigated in the pioneering work of Ruben \cite{ruben}, Kotz et al. \cite{kotz1,kotz2} when studying the distribution of random variables that are quadratic forms of independent normal variables. Even in small dimension 
it has important aplications
in Astronautics where for instance $\mu(\om)$ can model the probability of collision between two spatial
vehicles and must be computed with very good accuracy; see for instance
the works by Alfano \cite{alfano}, Chan \cite{chan}, Patera \cite{patera} and the more recent
\cite{spatial-1} which combines Laplace transform techniques 
with the theory of $D$-finite functions (see e.g. Zeilberger \cite{spatial-3} and Salvy \cite{spatial-4}). In doing so one obtains 
$\mu(\om)$ 
as a series with only nonnegative terms so that its evaluation does not involve error prone cancellations.
For more details the reader is referred to \cite{alfano,chan,patera,spatial-1} and the references therein. Other applications in Astronautics and for weapon evaluation
require to compute integral of bivariate Normal distributions on {\it convex polygons}, a case treated
in Didonato et al. \cite{polygons} where at an intermediate step the authors also evaluate Gaussian integrals on (unbounded) angular regions.

However the techniques developed in the above cited works 
and in some of the references therein, are not reproducible for more general sets $\om$ of the form (\ref{setk}).

It should be noted that computing exactly the Lebesgue measure of a compact convex set is a very difficult problem in general (even for a convex polytope) and in fact, even approximating the volume of a polytope within some bounds is difficult; on the other hand some {\it probabilistic} methods can provide good estimates. For more details on the computational complexity of volume computation
the interested reader is referred to the discussion in \cite{sirev} and the references therein. 

\subsection*{Contribution}

The goal of this paper is to provide a systematic numerical scheme to
approximate $\mu(\om)$ as closely as desired at the price of solving a hierarchy of semidefinite programs\footnote{A semidefinite program is a convex conic optimization problem which can be solved
at arbitrary recision (fixed in advance) efficiently, i.e., in time polynomial in its input size; for more details the interested reader is referred to e.g. \cite{handbook}.} 
of increasing size.
To do so we use  that computing $\mu(\om)$ is equivalent to solving a linear program on an appropriate space of measures,
an instance of the Generalized Problem of Moments. We also use that 
{\it all moments $\y=(y_\alpha)$, $\alpha\in\N^n$, of $\mu$ can be computed efficiently}. 
In fact, the methodology that we propose is valid for {\it any} measure $\mu$ on $\R^n$ 
which satisfies Carleman's condition and whose moments can be computed. (The gaussian and exponential
measures being two notable examples of such measures as their moments are available in closed form.)

These two ingredients were already used in Henrion et al. \cite{sirev} to compute 
the Lebesgue {\it volume} of a {\it compact} basic semi-algebraic set $\om$ as in (\ref{setk}), by solving a
hierarchy of semidefinite programs. The resulting sequence of optimal values was shown to converge to
the Lebesgue volume of $\om$. A technique to improve the convergence was also
provided in \cite{sirev} to limit  the Gibbs effect present in the initial and basic version of the numerical scheme.
(With this technique the convergence is indeed much faster but not monotone any more.)
So the novelty in this paper with respect to \cite{sirev} is threefold:

$\bullet$ We here consider the Gaussian and exponential measures on $\R^n$ and $\R^n_+$ respectively, instead of the Lebesgue measure. If
the latter is appropriate in applications of computational geometry (e.g. to compute ``volumes"), the former
are particularly interesting for applications  in Probability \& Statistics (especially the Gaussian measure
which plays a central role).
In addition, as already mentioned the method also works for {\it any} measure $\mu$ on $\R^n$ 
which satisfies Carleman's condition and whose moments can be computed.  In a few words,
we show how from the sole knowledge of moments of such measures $\mu$, one may approximate as closely as desired 
the measure $\mu(\om)$ of basic closed semi-algebraic sets $\om\subset\R^n$.

$\bullet$ We can handle {\it any} basic semi-algebraic set $\om$ of the form (\ref{setk}) whereas \cite{sirev} is restricted to compact ones. Importantly, the method provides two monotone sequences $\overline{\omega}_d$
and $\underline{\omega}_d$, $d\in\N$, such that $\underline{\omega}_d\leq\mu(\om)\leq \overline{\omega}_d$ for all $d$, 
and each sequence converges to $\mu(\om)$ as $d\to\infty$.
Therefore at each step $d$ the error $\overline{\omega}_d-\mu(\om)$ is bounded by 
$\epsilon_d:=\overline{\omega}_d-\underline{\omega}_d$
and $\epsilon_d\to0$ as $d\to\infty$. In our opinion this 
is an important feature which to the best of our knowledge is not present in other methods of the Literature, at least at this level of generality.

It is also worth noting that the method also provides a convergent numerical scheme to
approximate as closely as desired any ({\it \`a priori} fixed) number of moments of the measure $\mu_\om$,
the restriction of $\mu$ to $\om$ ($\mu(\om)$ being only the mass of $\mu_\om$).

$\bullet$ The acceleration technique described in \cite{sirev} also applies to our context of the Gaussian measure on non-compact sets $\om$. But again the monotone convergence is lost. Therefore we also provide another technique of independent interest to accelerate the convergence of the numerical scheme.
It uses Stokes' Theorem for integration which permits to obtain linear relations between moments of $\mu_\om$. 

\subsection*{Comptational remarks} As we did in \cite{sirev} for the compact case and the Lebesgue measure, the problem is formulated as an instance of the the {\it generalized problem of moments} with polynomial data and we approximate $\mu(\om)$ as closely as desired by solving
a hierarchy of semidefinite programs of increasing size. 
This procedure is implemented in the software package GloptiPoly\footnote{GloptiPoly \cite{gloptipoly} is a software package for solving the Generalized Problem of Moments with polynomial data.} which for modelling convenience uses the standard basis of monomials. 
As the monomial basis is well-known to be  a source of numerical ill-conditioning,
only a limited control on the output accuracy is possible and so in this case only bounds $(\overline{\omega}_d,\underline{\omega}_d)$ 
with $d\leq10$ for $n=2,3$ are meaningful.
Therefore for simple sets like rectangles (\ref{rectangle}) and ellipsoids (with $n=3$), in its present form our technique does not compete 
in terms of accuracy with {\it ad-hoc} procedures
like in e.g. Genz \cite{genz1} (rectangles) and the recent \cite{spatial-1} (for ellipsoids).
However, in
view of the growing interest for semidefinite programming 
and its use in many applications, it is expected that more efficient packages
will be available soon. For instance the semidefinite package SDPA \cite{sdpa} has now
been provided with a double precision variant\footnote{The SDPA-GMP, SDPA-QD and SDPA-DD  versions of the standard SDPA package; see {\tt http://sdpa.sourceforge.net/family.html\#sdpa}.}. Moreover  a
much better accuracy could be obtained if one uses other bases for polynomials than the standard monomial basis.
For instance when $\mu$ is the Gaussian measure one should rather use the basis of Hermite polynomials, orthogonal with respect to $\mu$. 

Such issues are beyond the scope of the present paper and on the other hand, for very general sets like (\ref{setk}) in relatively small dimension, and to the best of our knowledge for the first time, our method can provide relatively good bounds $\underline{\omega}_d\leq\mu(\om)\leq\overline{\omega}_d$ in a reasonable amount of time, and with strong theoretical guarantees as $d$ increases.

\section{Main result}

\subsection{Notation and definitions}
Let $\R[\x]$ be the ring of polynomials in the variables
$\x=(x_1,\ldots,x_n)$.
Denote by $\R[\x]_d\subset\R[\x]$ the vector space of
polynomials of degree at most $d$, which forms a vector space of dimension $s(d)={n+d\choose d}$, with e.g.,
the usual canonical basis $(\x^\alpha)$ of monomials.
Also, let $\N^n_d:=\{\alpha\in\N^n\,:\,\sum_i\alpha_i\leq d\}$ and
denote by $\Sigma[\x]\subset\R[\x]$ (resp. $\Sigma[\x]_d\subset\R[\x]_{2d}$)
the space of sums of squares (SOS) polynomials (resp. SOS polynomials of degree at most $2d$). 
If $f\in\R[\x]_d$, write
\[f(\x)\,=\,\sum_{\alpha\in\N^n_d}f_\alpha\, \x^\alpha\quad\left(= \sum_{\alpha\in\N^n_d}f_\alpha \,\x_1^{\alpha_1}\cdots \x_n^{\alpha_n}\right),\]
in the canonical basis and
denote by $\f=(f_\alpha)\in\R^{s(d)}$ its vector of coefficients. Finally, let $\s^n$ denote the space of 
$n\times n$ real symmetric matrices, with inner product $\la \A,\B\ra ={\rm trace}\,\A\B$, and where the notation
$\A\succeq0$ (resp. $\A\succ0$) stands for $\A$ is positive semidefinite. 

Let $(\A_j)$, $j=0,\ldots,s$,  be a set of real symmetric matrices. An inequality of the form
\[\left(\A(\x)\,:=\,\right)\: \A_0+\sum_{k=1}^s\A_k\,x_k\:\succeq0,\quad \x\in\R^s,\]
is called a {\it Linear Matrix Inequality} (LMI) and a set of the form $\{\x: \A(\x)\succeq0\}$ is
the canonical form of the feasible set of semidefinite programs.

Given a real sequence $\z=(z_\alpha)$, $\alpha\in\N^n$, define the Riesz linear functional $L_\z:\R[\x]\to\R$ by:
\[f\:(=\sum_\alpha f_\alpha\x^\alpha)\quad\mapsto L_\z(f)\,=\,\sum_{\alpha}f_\alpha\,z_\alpha,\qquad f\in\R[\x].\]
A sequence $\z=(z_\alpha)$, $\alpha\in\N^n$, has a representing measure $\mu$ if
\[z_\alpha\,=\,\int_{\R^n}\x^\alpha\,d\mu,\qquad\forall\,\alpha\in\N^n.\]

\subsection*{Moment matrix}
The {\it moment} matrix associated with a sequence
$\z=(z_\alpha)$, $\alpha\in\N^n$, is the real symmetric matrix $\M_d(\z)$ with rows and columns indexed by $\N^n_d$, and whose entry $(\alpha,\beta)$ is just $z_{\alpha+\beta}$, for every $\alpha,\beta\in\N^n_d$. 
Alternatively, let
$\v_d(\x)\in\R^{s(d)}$ be the vector $(\x^\alpha)$, $\alpha\in\N^n_d$, and
define the matrices $(\B_\alpha)\subset\s^{s(d)}$ by
\begin{equation}
\label{balpha}
\v_d(\x)\,\v_d(\x)^T\,=\,\sum_{\alpha\in\N^n_{2d}}\B_\alpha\,\x^\alpha,\qquad\forall\x\in\R^n.\end{equation}
Then $\M_d(\z)=\sum_{\alpha\in\N^n_{2d}}z_\alpha\,\B_\alpha$.
If $\z$ has a representing measure $\mu$ then
$\M_d(\z)\succeq0$ because
\[\langle\f,\M_d(\z)\f\rangle\,=\,\int f^2\,d\mu\,\geq0,\qquad\forall \,\f\,\in\R^{s(d)}.\]
A measure whose all moments are finite is said to be {\it moment determinate} if there is no other measure with same moments.
The support of a Borel measure $\mu$ on $\R^n$ (denoted $\supmu$) is the smallest closed set $\K$ such that $\mu(\R^n\setminus\K)=0$.

A sequence $\z=(z_\alpha)$, $\alpha\in\N^n$, satisfies Carleman's condition if
\begin{equation}
\label{carleman}
\sum_{k=1}^\infty L_\z(x_i^{2k})^{-1/2k}\,=\,+\infty,\qquad\forall i=1,\ldots,n.
\end{equation}
If a sequence $\z=(z_\alpha)$, $\alpha\in\N^n$, satisfies Carleman's condition (\ref{carleman}) and $\M_d(\z)\succeq0$ for all $d=0,1,\ldots$, then
$\z$ has a representing measure on $\R^n$ which is moment determinate; see e.g. \cite[Proposition 3.5]{lass-book-icp}.
In particular a sufficient condition for a measure $\mu$ to satisfy Carleman's condition 
is that $\int \exp(c\sum_i\vert x_i\vert)d\mu <\infty$ for some $c>0$.

\subsection*{Localizing matrix}
With $\z$ as above and $g\in\R[\x]$ (with $g(\x)=\sum_\gamma g_\gamma\x^\gamma$), the {\it localizing} matrix associated with $\z$ 
and $g$ is the real symmetric matrix $\M_d(g\,\z)$ with rows and columns indexed by $\N^n_d$, and whose entry $(\alpha,\beta)$ is just $\sum_{\gamma}g_\gamma z_{\alpha+\beta+\gamma}$, for every $\alpha,\beta\in\N^n_d$.
Alternatively, let $\C_\alpha\in\s^{s(d)}$ be defined by:
\begin{equation}
\label{calpha}
g(\x)\,\v_d(\x)\,\v_d(\x)^T\,=\,\sum_{\alpha\in\N^n_{2d+{\rm deg}\,g}}\C_\alpha\,\x^\alpha,\qquad\forall\x\in\R^n.\end{equation}
Then $\M_d(g\,\z)=\sum_{\alpha\in\N^n_{2d+{\rm deg}g}}z_\alpha\,\C_\alpha$.
If $\z$ has a representing measure $\mu$ whose support is 
contained in the set $\{\x\,:\,g(\x)\geq0\}$ then
$\M_d(g\,\z)\succeq0$ because
\[\langle\f,\M_d(g\,\z)\f\rangle\,=\,\int f^2\,g\,d\mu\,\geq0,\qquad\forall \,\f\,\in\R^{s(d)}.\]
Let $g_0(\x)=1$ for all $\x\in\R^n$, and with a family $(g_1,\ldots,g_m)\subset\R[\x]$ is associated 
the {\it quadratic module}:
\begin{equation}
\label{module}
Q(g_1,\ldots,g_m)\,:=\,\left\{\sum_{j=0}^m\sigma_j\,g_j:\quad \sigma_j\in\Sigma[\x],\:j=0,\ldots,m\,\right\},
\end{equation}
and its {\it truncated} version
\begin{equation}
\label{truncated-module}
Q_k(g_1,\ldots,g_m)\,:=\,\left\{\sum_{j=0}^m\sigma_j\,g_j:\quad \sigma_j\in\Sigma[\x]_{k-d_j},\:j=0,\ldots,m\,\right\},
\end{equation}
where $d_j=\lceil{\rm deg}(g_j)/2\rceil$, $j=0,\ldots,m$.
Next, given a closed set $\mathcal{X}\subseteq\R^n$, let $C(\mathcal{X})\subset\R[\x]$ (resp. $C_d(\mathcal{X})\subset\R[\x]_d$) be the convex cone of polynomials (resp. polynomials of degree at most $2d$) that are nonnegative on $\mathcal{X}$.

For mored details on the above notions of moment and localizing matrix, quadratic module,
as well as their use in potential applications, the interested reader is referred to Lasserre \cite{lass-book-icp}.

\subsection{A basic numerical scheme for upper bounds}

Recall that the goal is to compute $\mu(\om) $, where 
$\mu$ is the Gaussian measure in (\ref{mu-gauss}) or the exponential measure in (\ref{mu-exp})
and $\om$  the basic semi-algebraic set in (\ref{setk}). 
Let $\y=(y_\alpha)$, $\alpha\in\N^n$, be the moments of $\mu$, i.e.,
\[y_\alpha\,=\,\int \x^\alpha\,d\mu,\qquad\alpha\in\N^n,\]
which are obtained in closed form.
Call $\mu_\om$ the {\em restriction} of $\mu$ to $\om$, that is:
\[\mu_\om(B)\,=\,\mu(\om\cap B),\quad\forall B\in\,\mathcal{B}(\R^n).\]
Let $M(\R^n)$ (resp. $M(\om)$) be the convex cone of finite (nonnegative) Borel measures 
on $\R^n$ (resp. $\om$), and consider the new abstract problem:
\begin{equation}
\label{abstract-primal-new}
\rho\,=\,\sup_{\phi,\nu}\:\{\,\int f\,d\phi:\: \phi+\nu\,=\,\mu;\quad\phi\in M(\om) ,\:\nu\in M(\R^n)\,\}.
\end{equation}
\begin{lem}
\label{lem-positive}
Let $f\in\R[\x]$ be strictly positive almost everywhere on $\om$ . Then 
$(\phi^*,\nu^*):=(\mu_\om,\mu-\mu_\om)$  is the unique optimal solution of (\ref{abstract-primal-new}) and 
$\phi^*(\om)=\mu(\om)$ while $\rho=\int_\om fd\mu$. In particular if $f=1$ then $\rho=\mu(\om)$.
\end{lem}
\begin{proof}
As $\phi\leq\mu$ and $f\geq0$ on $\om$ , we have $\int f\,d\phi\leq\int_{\om  }f \, d\mu=\int fd\mu_\om  $, and so
$\rho\leq \int fd\mu_\om  $. On the other hand, pick $(\phi^*,\nu^*):=(\mu_\om  ,\mu-\mu_\om) $, to obtain 
$\int fd\phi^*=\int fd\mu_\om  $, which shows that $\rho=\int fd\mu_\om $, and $\phi^*(\om)=\mu(\om)$. 

Next, as $\phi+\nu=\mu$
every feasible solution is such that $\phi$ is absolutely continuous with respect to $\mu$, and
hence with respect to $\mu_\om  $, denoted $\phi\ll\mu_\om  $. In particular, by the Radon-Nikodym theorem \cite{ash}, there exists a 
nonnegative measurable function $g$ on $\om$  such that
\[\phi(B)\,=\,\int_B g(\x)\,d\mu_\om  ,\qquad\forall B\in\mathcal{B}(\om) ,\]
and since $\phi(B)\leq \mu_\om  (B)$ for all $B\in\mathcal{B}(\om) $, it follows that
$g(\x)\leq 1$ almost everywhere on $\om$ .
So suppose that there exists another optimal solution $(\phi,\nu)$ with $\int fd\phi=\rho$, that is,
$\int  f (g-1)d\mu_\om  =0$. As $f>0$  a.e. on $\om$ and $g\leq 1$ on $\om$, this implies $\mu_\om  (\{\x:g(\x)<1\})=0$.
In other words $g=1$ a.e. on $\om$ which in turn implies $\phi(B)=\mu(B)=\phi^*(B)$ for all
$B\in\mathcal{B}(\om)$. That is, $(\phi,\nu)=(\phi^*,\nu^*)$.
\end{proof}
Problem (\ref{abstract-primal-new}) is an infinite dimensional linear program on an appropriate space of measures and
its dual is the following optimization problem:
\begin{equation}
\label{abstract-dual}
\rho^*\,=\,\displaystyle\inf_{p\in\R[\x]}\,\{\,\displaystyle \int  p\,d\mu:\: p-f\in C(\om); \quad p\in\,C(\R^n)\,\}
\end{equation}
with $\rho^*\geq\int fd\mu_\om$. However neither (\ref{abstract-primal-new}) nor its dual (\ref{abstract-dual})
are tractable. This is  why, with $d\geq d_0:=\max_jd_j$ fixed, we now consider the semidefinite program:
\begin{equation}
\label{scheme-1-primal}
\begin{array}{rl}
\ov_d\,=\,\displaystyle\sup_{\u,\v\in\R[\x]_{2d}^*}&\{\,L_\u(f)\:\mbox{ s.t. }
u_\alpha+v_\alpha\,=\,y_\alpha,\quad\forall\,\alpha\in\N^n_{2d};\\
&\M_d(\u),\,\M_d(\v)\succeq0;\:\M_{d-d_j}(g_j\,\u)\succeq0,\quad j=1,\ldots,m\},
\end{array}
\end{equation}
which is a {\it relaxation} of  (\ref{abstract-primal-new}), hence with $\overline{o}_d\geq\int_\om fd\mu$ for all $d$. 
Indeed the constraints on $(\u,\v)$ in (\ref{scheme-1-primal}) are only necessary
for $\u$ and $\v$ to be moments of some measures $\phi$ on $\om$ and $\nu$ on $\R^n$ respectively,
such that $\phi+\nu=\mu$. 

\subsection*{Computational complexity}
The semidefinite program (\ref{scheme-1-primal}) has $2\,{n+2d\choose n}$ variables and the 
matrices $\M_d(\u),\M_d(\v)$ are ${n+d\choose n}\times {n+d\choose n}$. In principle most SDP solvers
use interior point methods, considered to be the most efficient. 
So in view of the present status of state-of-the-art semidefinite solvers,
this methodology is still limited to problems in small dimension $n$ (say $n\leq 3,4$). Indeed even though 
semidefinite programs (SDPs) are convex optimization problems, the computational complexity of interior point methods
is still too high to solve large size SDPs on today's computers. For more details on interior point methods, their analysis and refinements, see e.g.  \cite{handbook} and \cite{wright}.

\begin{rem}
\label{support}
In case where $\mu$ is the exponential measure (\ref{mu-exp}) then the support of the measure $\phi$ in (\ref{scheme-1-primal})
is in fact $\om\cap\supmu=\om\cap\R^n_+$, which is the same as replacing $\om$ with $\om\cap\R^n_+$, i.e.,
\begin{equation}
\label{newsetk}
\om\,:=\,\om\,\cap\,\R^n_+\,=\,\{\x: g_j(\x)\geq0,\:j=1,m+n\},
\end{equation}
by introducing the additional polynomials $\x\mapsto g_{m+i}(\x)=x_i$, $i=1,\ldots,n$. 

Then in (\ref{scheme-1-primal}) the ``support" constraints
$\M_{d-d_j}(g_j\,\u)\succeq0$, $j=1,\ldots,m$, now become
$\M_{d-d_j}(g_j\,\u)\succeq0$, $j=1,\ldots,m+n$.
\end{rem}

The dual of (\ref{scheme-1-primal}) is also a semidefinite program, which in compact form reads:
\begin{equation}
\label{scheme-1-dual}
\ov^*_d\,=\,\displaystyle\inf_{p\in\R[\x]_{2d}}\,\{\,\displaystyle \int  p\,d\mu:\: p-f\in Q_{2d}(\om); \quad p\in\Sigma[\x]_d\,\},
\end{equation}
and as expected, is a {\it strengthening} of (\ref{abstract-dual}). Indeed one has replaced the nonnegativity contraint
$p-f\geq0$ on $\om$ (resp. $p\geq0$ on $\R^n$) with the stronger membership $(p-f)\in Q_{2d}(\om)$
(resp. the stronger membership $p\in\Sigma[\x]_d$).\\

One main result of this paper is the following:
\begin{thm}
\label{th1}
Let $\mu$ be as in (\ref{mu-gauss}) or (\ref{mu-exp}).
With $\om  \subset\R^n$ as in (\ref{setk}) (or as in (\ref{newsetk}) when $\mu$ is the exponential measure), assume that
both $\om$  and $\supmu\setminus\om  $ have nonempty interior and let $f\in\R[\x]$ be strictly positive 
a.e. on $\om$. Then:

(a) Both problems (\ref{scheme-1-primal}) and 
(\ref{scheme-1-dual}) have an optimal solution and $\ov_d=\ov^*_d$ for every $d\geq d_0$.

(b) The sequence $(\ov_d)$, $d\in\N$, is monotone non-increasing and $\ov_d\to\int_\om fd\mu$ as $d\to\infty$.
In addition if $(\u^d,\v^d)$ is an optimal solution of (\ref{scheme-1-primal}) then $u^d_0\to\mu(\om)$ as $d\to\infty$.

(c) If $f=1$ then the sequence $(\ov_d)$, $d\in\N$, is monotone non-increasing and $\ov_d\to\mu(\om) $ as $d\to\infty$.
\end{thm}
\begin{proof}
(a) Recall that $\mu_\om  $ is the restriction of $\mu$ to $\om$,
and let $\u=(y^\om  _\alpha)$), $\alpha\in\N^n_{2d}$, be the moments of $\mu_\om$,
up to order $2d$. Similarly let $\v=(y_\alpha-u_\alpha)$), $\alpha\in\N^n_{2d}$, be the moments of $\mu-\mu_\om  $.
Hence $(\u,\v)$ is a feasible for (\ref{scheme-1-primal}). 
In addition,
as both $\om$  and $\supmu\setminus\om  $ have nonempty interior,  $\M_d(\u)\succ0$ and $\M_d(\v)\succ0$.
Indeed otherwise there would exists $0\neq p\in\R[\x]_d$ such that 
\[0\,=\,\langle \mathbf{p},\M_d(\u)\,\mathbf{p}\rangle\,=\,   \int p^2d\mu\quad\Rightarrow \quad p=0\:\mbox{ a.e. on $\om$.}\]
But then $p=0$ since it has to vanish on some open set of $\om$. The same argument also works for $\v$.
Hence Slater's condition holds for (\ref{scheme-1-primal}) since $(\u,\v)$ is  strictly feasible solution. By a standard result in convex optimization,
this in turn implies $\ov_d=\ov_d^*$ and moreover (\ref{scheme-1-dual}) has an optimal solution if $\ov_d^*$ is finite. We next prove that
(\ref{scheme-1-primal}) has an optimal solution and therefore so does (\ref{scheme-1-dual}). Observe that from the constraint
$u_\alpha+v_\alpha=y_\alpha$ we deduce that $u_0\leq y_0$, $v_0\leq y_0$. Moreover all (nonnegative) diagonal elements of $\M_d(\u)$
and $\M_d(\v)$ are dominated by those of $\M_d(\y)$. In particular:
\[\max_i[\,L_\u(x_i^{2d})\,]\,\leq\,\max_i[\,L_\y(x_i^{2d})\,];\quad\max_i[\,L_\v(x_i^{2d})\,]\,\leq\,\max_i[\,L_\y(x_i^{2d})\,].\]
From Lasserre \cite[Proposition 2.38]{lass-camb},
$\vert u_\alpha\vert\leq\tau_d$ and $\vert v_\alpha\vert\leq\tau_d$, for all $\alpha\in\N^n_{2d}$
(where $\tau_d:=\max[y_0,\max_i[\,L_\y(x_i^{2d})\,]$). Therefore the (closed) feasible of (\ref{scheme-1-primal}) is bounded, hence compact,
which in turn implies that (\ref{scheme-1-primal}) has an optimal solution. 

(b) That the sequence $(\ov_d)$, $d\in\N$, is monotone non-increasing is straightforward. Next, let $(\u^d,\v^d)$ be an optimal solution of (\ref{scheme-1-primal}). 
Notice that from the proof of (a) we have seen that
$\vert u^d_\alpha\vert\leq \tau_d$ and $\vert v^d_\alpha\vert\leq \tau_d$ for all $\alpha\in\N^n_{2d}$. Completing with zeros
we now consider $\u^d$ and $\v^d$ as infinite vectors indexed by $\N^n$.
Therefore, for each infinite sequence $\u^d=(u^d_\alpha)$, $\alpha\in\N^n$, it holds:
\[u^d_0\leq y_0\quad\mbox{ and }\quad 2j-1\leq\vert\alpha\vert\leq 2j \:\Rightarrow\:\vert u^d_\alpha\vert\,\leq\,\tau_j, \quad\forall j=1,\ldots,\]
Similarly
\[v^d_0\leq y_0\quad\mbox{ and }\quad 2j-1\leq\vert\alpha\vert\leq 2j \:\Rightarrow\:\vert v^d_\alpha\vert\,\leq\,\tau_j, \quad\forall j=1,\ldots,\]
Hence by  a standard argument there exists a subsequence $d_k$ and two infinite sequences
$\u^*=(u^*_\alpha)$ and $\v^*=(v^*_\alpha)$, $\alpha\in\N^n$, such that
\begin{equation}
\label{convergence}
\mbox{for every $\alpha\in\N^n$,}\quad u^{d_k}_\alpha\to u^*_\alpha\quad\mbox{and}\quad v^{d_k}_\alpha\to v^*_\alpha, \quad\mbox{as $k\to\infty$.}
\end{equation}
In particular $\lim_{k\to\infty}\ov_{d_k}=\lim_{k\to\infty}L_{\u^{d_k}}(f)=L_{\u^*}(f)$, and as the sequence 
$(\ov_d)$, $d\in\N$, is monotone:
\begin{equation}
\label{aux33}
\int_\om fd\mu \,\leq\,\lim_{d\to\infty}\ov_d\,=\,\lim_{k\to\infty}\ov_{d_k}\,=\,L_{\u^*}(f)
\quad\mbox{and}\quad u^*_\alpha+v^*_\alpha\,=\,y_\alpha,\qquad\forall\,\alpha\in\N^n.
\end{equation}

Then the convergence (\ref{convergence}) implies that for each fixed $k$, $\M_k(\u^*)\succeq0$, $\M_k(\v^*)\succeq0$, and 
$\M_k(g_j\,\u^*)\succeq0$ for all $j=1,\ldots,m$. Next observe that the moment sequence $\y$
of the Gaussian or exponential measure $\mu$ satisfies Carleman's condition (\ref{carleman}) and so
$\mu$ is moment determinate.
But from $L_{\u^*}(x_i^{2d})\leq L_{\y}(x_i^{2d})$ and $L_{\v^*}(x_i^{2d})\leq L_{\y}(x_i^{2d})$,
we deduce that both $\u^*$ and $\v^*$ also satisfy Carleman's condition.
This latter fact combined with $\M_k(\u^*)\succeq0$ and $\M_k(\v^*)\succeq0$ 
for all $k$, yields that $\u^*$ and $\v^*$ are the moment sequences of some
finite Borel measures $\phi^*$ and $\nu^*$ on $\R^n$; 
see for instance Lasserre \cite[Proposition 3.5]{lass-book-icp}

Next, as $\u^*$ satisfies $L_{\u^*}(x_i^{2d})\leq L_{\y}(x_i^{2d})$ and $\mu$ is the Gaussian or exponential measure,
there is some $M>0$ such that $L_\u(x_i^{2k})\leq M(2k){\rm !}$ for all $k$. 
As in addition $\M_{k}(g_j\,\u^*)\succeq0$, for all $k$, and all $j=1,\ldots,m$,
by Lasserre \cite[Theorem 2.2, p. 2494]{lass-tam},
the support of $\phi^*$ is contained in $\om$. Hence $\phi^*\in M(\om) $ and $\nu^*\in M(\R^n)$.
Moreover from (\ref{aux33}), 
\[\int  \x^\alpha \,d(\phi^*+\nu^*)\,=\,\int \x^\alpha\,d\mu,\qquad \forall\,\alpha\in\N^n,\]
and as $\mu$ is moment determinate it follows that $\phi^*+\nu^*=\mu$.
Hence the pair $(\phi^*,\nu^*)$ is feasible for problem (\ref{abstract-primal-new}) with value
$\int fd\phi^*=L_{\u^*}(f)\geq\int_\om fd\mu=\rho$, which proves that $(\phi^*,\nu^*)$ is an optimal solution
of problem (\ref{abstract-primal-new}), and so $\int fd\phi^*=\int_\om fd\mu$.

(c) When $f=1$ then $\ov_d=u^d_0$ and by (\ref{convergence}) $u^{d_k}_0\to\mu(\om)$ as $k\to\infty$.
Therefore by monotonicity of the sequence $(\ov_d)$, the result follows.
\end{proof}

\begin{rem}
{\rm (i) The fact that $\om$ is a basic semi-algebraic set of the form $\{\x: g_j(\x)\geq0,\:j=1,\ldots,m\}$ is exploited in the construction
of the semidefinite relaxation (\ref{scheme-1-primal}) through the psd constraints  $\M_{d-d_j}(g_j\,\u)\succeq0$, $j=1,\ldots,m$, on the localizing matrices associated with the $(g_j)$ that define $\om$.

(ii) We emphasize that the set $\om$ is {\em not} assumed to be compact. So Theorem \ref{th1} extends 
the methodology proposed in \cite{sirev} in two directions: Firstly, we now consider the Gaussian or exponential measure
instead of the Lebesgue measure and secondly, 
we now allow arbitrary (possibly non-compact) basic semi-algebraic sets $\om$ 
(whereas in \cite{sirev} $\om$  has to be compact) .

(iii) Solving (\ref{scheme-1-dual}) has a simple interpretation: Namely when $f=1$
one tries to approximate the indicator function $\x\mapsto 1_\om  (\x)$ ($=1$ if $\x\in\om  $ and $0$ otherwise)
by polynomials of increasing degree. It is well-known that a Gibbs effect occurs at points of the boundary of $\om$ .
So if we choose $\x\mapsto f(\x):=\prod_{j=1}^mg_j(\x)$ (so that $f$ is continuous, nonnegative on $\om$ , and 
vanishes on the boundary of $\om$) then the function $\x\mapsto f(\x)1_\om  (\x)$ is continuous, hence much easier
to approximate by polynomials. For compact sets $\om$ and 
the Lebesgue measure $\mu$, it has been observed on \cite{sirev} that this strategy strongly limits the Gibbs effect
and yields drastic improvements on the the convergence to $\mu(\om)$ (but this convergence is not monotone any more). 

However, when $f=1$ the monotone convergence $\ov_d\to\mu(\om)$ as $d\to\infty$ is 
an important attractive feature of the method because when convergence 
has not taken place yet, one still has the useful information that $\mu(\om)\leq \ov_d$,
which is important in some applications.
}
\end{rem}
We next show how even with $f=1$ one may still improve the convergence $\ov_d\to \mu(\om)$ significantly
(and thus still keep  a monotone sequence of upper bounds on $\mu(\om)$).
But before we show how to get a converging sequence of lower bounds on $\mu(\om)$ with same techniques.

\subsection{Lower bounds on $\mu(\om)$}
\label{lower}
Recall that if $\mu$ is the exponential measure then for practical computation in (\ref{scheme-1-primal}) one 
replaces $\om$ with $\om:=\om\cap\supmu=\om\cap\R^n_+$, that is,
$\om$ is defined in (\ref{newsetk}); see Remark \ref{support}.

Let $\om^c:=\supmu\setminus\om$ be the complement of $\om$ in the support of $\mu$.
As $\mu$ is absolutely continuous with respect to the Lebesgue measure,
the set $\om^c$ can be written as 
\begin{equation}
\label{complement}
\om^c\,=\,\displaystyle\cup_{\ell=1}^s\om^c_\ell \quad\mbox{with}\quad
\mu(\om^c)\,=\,\sum_{\ell=1}^s \mu(\om_\ell^c),
\end{equation}
where each $\om^c_\ell$ is a closed basic semi-algebraic set, $\ell=1,\ldots,s$,
and the overlaps between the sets $\om^c_\ell$ have $\mu$-measure zero.
Of course the decomposition
(\ref{complement}) is not unique.
For instance if $\mu$ is the Gaussian measure and
$\om=\{\x: g_j(\x)\geq0,\,j=1,2\}$ then $\om^c=\om^c_1\cup\om^c_2$ with:
\[\om^c_1\,=\,\{\x: g_1(\x)\,\leq\,0\,\};\quad \om^c_2\,:=\,\{\x: g_1(\x)\,\geq\,0;\:g_2(\x)\leq0\}.\]
On the other hand if $\mu$ is the exponential measure and
$\om=\{\x\geq0: g_j(\x)\geq0,\,j=1,2\}$ then $\om^c=\om^c_1\cup\om^c_2$ with:
\[\om^c_1\,=\,\{\x\geq0: g_1(\x)\,\leq\,0\,\};\quad \om^c_2\,:=\,\{\x\geq0: g_1(\x)\,\geq\,0;\:g_2(\x)\leq0\}.\]
Then write 
\begin{equation}
\label{def-olc}
\om^c_\ell\,=\,\{\x:\:g_{\ell j}(\x)\geq0,\:j=1,\ldots,m_\ell\,\},\qquad \ell=1,\ldots,s.
\end{equation}
for some integer $m_\ell$ and some polynomials $(g_{\ell j})\subset\R[\x]$, $j=1,\ldots,m_\ell$.
Again let $d_j:=\lceil({\rm deg} \,g_{\ell j})/2\rceil$ and $d_0=\max_j d_j$.
\begin{cor}
\label{cor-1}
Let $\om  \subset\R^n$ be as in (\ref{setk}) (or as in (\ref{newsetk}) when $\mu$ is the exponential measure). Assume that
both $\om$  and $\om^c$ have nonempty interior and let $\om^c$ be as in (\ref{complement}).

(a) If $f=1$ then for each $\ell=1,\ldots,s$, let $\ov^c_{\ell d}$ be the optimal value of the semidefinite 
program (\ref{scheme-1-primal}) with $g_{\ell j}$ in lieu of $g_j$ (and $m_\ell$ in lieu of $m$), and
for every $d\geq d_0$, let :
\begin{equation}
\label{cor-1-1}
\und_d\,:=\,\mu(\R^n)-\left(\sum_{\ell=1}^s\ov^c_{\ell d}\,\right).
\end{equation}
Then the sequence $\und_d$, $d\in\N$, is monotone non-decreasing
with $\mu(\om)\geq\und_d$ for all $d\geq d_0$ and 
$\und_d\to\mu(\om)$ as $d\to\infty$.

(b) If $f\neq 1$ and $f$ vanishes on $\partial\om$, then for each $\ell=1,\ldots,s$, let $\ov^c_{\ell d}$ be the optimal value of the semidefinite program (\ref{scheme-1-primal}) with $f^2$ and $g_{\ell j}$ in lieu of $f$ and $g_j$ (and $m_\ell$ in lieu of $m$), and for every $d\geq d_0$, let :
\begin{equation}
\label{cor-1-2}
\und_d\,:=\,\int f^2d\mu-\left(\sum_{\ell=1}^s\ov^c_{\ell d}\,\right).
\end{equation}
Then the sequence $\und_d$, $d\in\N$, is monotone non-decreasing
with $\int_\om f^2d\mu\geq\und_d$ for all $d\geq d_0$ and 
$\und_d\to\int_\om f^2d\mu$ as $d\to\infty$.
\end{cor}
\begin{proof}
(a) By construction of the semidefinite program (\ref{scheme-1-primal}),
$\ov^c_{\ell d}\geq \mu(\om^c_\ell)$ for every $d\geq d_0$ and every $\ell=1,\ldots,s$ and so
\[\sum_{\ell=1}^s \ov^c_{\ell d}\,\geq\,\sum_{\ell=1}^s \mu(\om^c_\ell)\,=\, \mu(\cup_\ell\om^c_\ell)\,=\,
\mu(\om^c),\]
which in turn implies 
\begin{equation}
\label{aux44}
\und_d\,=\,\mu(\R^n)-\sum_{\ell=1}^s \ov^c_{\ell d}\,\leq\,\mu(\R^n)-\mu(\om^c)\,=\,\mu(\om),\quad\forall d.\end{equation}
Next from Theorem \ref{th1}(b) one has $\ov^c_{\ell d}\to\mu(\om^c_\ell)$ as $d\to\infty$, for every $\ell=1,\ldots,s$. Therefore
using (\ref{complement}) and taking limit in (\ref{aux44}) as $d\to\infty$, yields the desired result. 

(b) Similarly, by construction of the semidefinite program (\ref{scheme-1-primal}),
$\ov^c_{\ell d}\geq \int_{\om^c_\ell} f^2d\mu$ for every $d\geq d_0$ and every $\ell=1,\ldots,s$ and so
\[\sum_{\ell=1}^s \ov^c_{\ell d}\,\geq\,\sum_{\ell=1}^s \int_{\om^c_\ell}f^2d\mu\,
\geq\,\int_{\cup_\ell \om^c_\ell}f^2d\mu\,=\, \int_{\om^c} f^2d\mu,\]
which in turn implies 
\begin{equation}
\label{aux444}
\und_d\,=\,\int f^2d\mu-\sum_{\ell=1}^s \ov^c_{\ell d}\,\leq\,\int f^2d\mu-\int_{\om^c}f^2d\mu\,=\,\int_\om f^2 d\mu,\quad\forall d.\end{equation}
Next from Theorem \ref{th1}(b) one has $\ov^c_{\ell d}\to\mu(\om^c_\ell)$ as $d\to\infty$, for every $\ell=1,\ldots,s$. Therefore
using (\ref{complement}) and taking limit in (\ref{aux44}) as $d\to\infty$, yields the desired result. 
\end{proof}

As a consequence of Theorem \ref{th1} and  Corollary \ref{cor-1}, we finally obtain:
\begin{cor}
\label{cor-upper-lower}
Let $\om$ be as in (\ref{setk}) (or as in (\ref{newsetk}) when $\mu$ is the exponential measure)
and $\ov_d$ be as in (\ref{scheme-1-primal}). If $f=1$ then
\begin{equation}
\label{conclusion-1}
\und_{d}\,\leq\, \mu(\om)\,\leq\, \ov_d,\quad \forall d\geq d_0\quad\mbox{and}\quad  
\lim_{d\to\infty}\,\und_{d}\,=\,\mu(\om)\,=\,\lim_{d\to\infty}\,\ov_{d}.
\end{equation}
where $\und_d$ is defined in (\ref{cor-1-1}).

Similarly let $f\in\R[\x]$ vanish on $\partial\om$. If
$\ov_d$ in (\ref{scheme-1-primal}) is computed with $f^2$ (in lieu of $f$) then
\begin{equation}
\label{conclusion-2}
\und_{d}\,\leq\, \int_\om f^2d\mu\,\leq\, \ov_d,\quad \forall d\geq d_0\quad\mbox{and}\quad  
\lim_{d\to\infty}\,\und_{d}\,=\,\int_\om f^2d\mu\,=\,\lim_{d\to\infty}\,\ov_{d}.
\end{equation}
where $\und_d$ is defined in (\ref{cor-1-2}).
\end{cor}
Why can (\ref{conclusion-2}) be also potentially interesting? Recall that any optimal solution
$\u^d$ of (\ref{scheme-1-primal}) is such that $u^d_0\to \mu(\om)$ as $d\to\infty$. So if in (\ref{conclusion-2})
$\ov_d-\und_d$ is small then it is a good indication that $u^d_0\approx\mu(\om)$.

\section{Improving convergence}
\label{method}
Corollary \ref{cor-upper-lower} states nothing on the rate of convergence for the sequences
$(\ov_d)$ and $(\und_d)$, $d\in\N$, of upper and lower bounds on $\mu(\om)$. 
As already observed in \cite{sirev} for compact sets $\om$
and the Lebesgue measure, when $f=1$ the convergence appears to be rather slow.
On the other hand, if $f$ is positive on $\om$ and vanishes on $\partial\om$
the convergence is significantly faster but one looses the non-increasing (resp. non-decreasing) monotonicity
of the sequence $\ov_d$ (resp. $\und_d$). Keeping the monotonicity is important
because no matter how close to $\mu(\om)$ is $\ov_d$ or $\und_d$,
one still has the useful information that $\und_d\leq\mu(\om)\leq\ov_d$ for all $d$.

In this section we show how to improve significantly this convergence while keeping the monotonicity
of the sequences of upper and lower bounds.
To do this we use Stokes' Theorem for integration and in the sequel, to avoid technicalities, we assume that 
$\om\subset\R^n$ is the closure of its interior, i.e., $\om=\overline{{\rm int}(\om)}$.

\subsection{Stokes helps}

As we already know in advance that $(\mu_\om  ,\mu-\mu_\om) $ is an optimal solution of the infinite-dimensional linear program
(\ref{abstract-primal-new}), any additional information on the moments of $\mu_\om  $
will be helpful if when included as additional constraints in the relaxation (\ref{scheme-1-primal}),
one still has a semidefinite program to solve.  Fortunately for basic semi-algebraic sets $\om$
this is possible. Indeed suppose for the moment that
$\om$  is compact with smooth boundary
$\partial\om$ and let $X$ be some given vector field.
Let $\theta_\mu$ be the density of the Gaussian or exponential measure $\mu$ and $f\in\R[\x]$
a given polynomial. (In case where $\mu$ is the exponential measure then one replaces $\om$ with $\om\cap\R^n_+$.)
By Stokes' Theorem, for every $\alpha\in\N^n$:
\[\int_\om   {\rm Div}(X)\,\x^\alpha\,f(\x)\,\theta_\mu(\x)\,d\x+\int_\om   \langle X,\nabla (\x^\alpha f(\x)\theta_\mu(\x))\rangle\,d\x\]
\[\,=\,\int_{\partial\om  } \langle X,\vec{n}_\x\rangle\, \x^\alpha f(\x)\,\theta_\mu(\x)\,d\sigma\]
(where $\vec{n}_\x$ is the outward pointing normal at $\x\in\partial\om $ and $\sigma$ is the $(n-1)$-Hausdorff measure
on $\partial\om$). In fact
the above identity holds even if the boundary is algebraic and not smooth everywhere\footnote{A generalization of Stokes' Theorem is valid for smooth differential forms (see Whitney \cite[Theorem 14A]{whitney}, or Federer \cite{federer}). Such generalization of Stokes' Theorem has already been used in other contexts like e.g. in Barvinok \cite{barvinok} to provide formulae for exponential integrals on polyhedra (hence with ``corners" on the boundary).}.
Therefore if $f\in\R[\x]$ vanishes on the boundary $\partial\om$ then
\begin{equation}
\label{stokes-theorem}
\int_\om   {\rm Div}(X)\,\x^\alpha\,f(\x)\,\theta_\mu(\x)\,d\x+\int_\om   \langle X,\nabla (\x^\alpha f(\x)\theta_\mu(\x))\rangle\,d\x\,=\,0.\end{equation}
So for instance with $X:=e_i$ (with $e_i=(\delta_{j=i})\in\R^n$) (and since ${\rm Div}(X)=0$):
\begin{equation}
\label{stokes1}
\int_\om   \frac{\partial (\x^\alpha \,f(\x)\,\theta_\mu(\x))}{\partial x_i}\,d\x\,=\,0,
\qquad\forall\,\alpha\in\N^n,
\end{equation}
Equivalently, introduce the polynomials $p_{i,\alpha}\in\R[\x]$, $i=1,\ldots,n$ (of degree 
$d_\alpha={\rm deg}\,f+\vert\alpha\vert+1$ if $\mu$ is the Gaussian measure and
$d_\alpha={\rm deg}\,f+\vert\alpha\vert$ if $\mu$ is the exponential measure),
defined by:
\begin{equation}
\label{palpha-gaussian}
\x\mapsto p_{i,\alpha}(\x)\,:=\,\frac{\partial (\x^\alpha f)}{\partial x_i}-\x^\alpha \,f(\x)\,x_i,
\qquad\forall\,\alpha\in\N^n,\end{equation}
if $\mu$ is the Gaussian measure (\ref{mu-gauss}) and by
\begin{equation}
\label{palpha-exp}
\x\mapsto p_{i,\alpha}(\x)\,:=\,\frac{\partial (\x^\alpha f)}{\partial x_i}-\x^\alpha \,f(\x)
\qquad\forall\,\alpha\in\N^n,\end{equation}
if $\mu$ is the exponential measure (\ref{mu-exp}). Then we have the equality constraints
\begin{equation}
\label{stokes2}
\int p_{i,\alpha}\,d\mu_\om  \,=\,0,\qquad \forall\,\alpha\in\N^n,\:i=1,\ldots,n,
\end{equation}
which defines a {\it linear} constraint on the moments of $\mu_\om$.

\subsection{Extension to non-compact sets $\om$}

We next show that (\ref{stokes1}) (or, equivalently (\ref{stokes2})) extends to
non-compact semi-algebraic sets $\om$. 
\begin{lem}
\label{non-compact}
Let $\om\subset\R^n$ be as in (\ref{setk}) (not necessarily compact) 
or as in (\ref{newsetk}) if $\mu$ is the exponential measure
(\ref{mu-exp}). Let $f\in\R[\x]$ vanish on the boundary
$\partial\om$.
Then with $p_{i,\alpha}\in\R[\x]$ as in (\ref{palpha-gaussian}),
\begin{equation}
\int_{\om}   \frac{\partial (\x^\alpha \,f\,\exp(-\Vert\x\Vert^2/2))}{\partial x_i}\,d\x\,=\,
\int_{\om}p_{i,\alpha}\,d\mu\,=\,0,
\end{equation}
for every $\alpha\in\N^n$ and $i=1,\ldots,n$. Similarly, with $p_{i,\alpha}\in\R[\x]$ as in (\ref{palpha-exp}),
\begin{equation}
\int_{\om}   \frac{\partial (\x^\alpha \,f\,\exp(-\sum_{i=1}^n x_i)}{\partial x_i}\,d\x\,=\,
\int_{\om}p_{i,\alpha}\,d\mu\,=\,0,
\end{equation}
for every $\alpha\in\N^n$ and $i=1,\ldots,n$. 
\end{lem}
\begin{proof}
We only consider the case of the Gaussian measure as the proof for the exponential measure uses similar
arguments.
For $M>0$ arbitrary, fixed, let $\B_{M}:=\{\x:\Vert\x\Vert^2\leq M^2\,\}$
and consider the compact set $\om_M:=\om\cap\B_{M}$.
Write $\partial\om_M:=\partial\om_M^1\cup\partial\om_M^2$, where
\[\partial\om_M^1\,:=\,\partial\om\,\cap\,\{\,\x:\Vert\x\Vert^2\leq M^2\,\};\quad
\partial\om_M^2\,:=\,\om\,\cap\,\{\,\x:\Vert\x\Vert^2\,=\,M^2\,\}.\]

Even though the boundary $\partial\om_M$ is not smooth everywhere, Stokes' Theorem still applies (see previous footnote) and so
if $X=e_i$ and if $f$ vanishes on $\partial\om$,
(\ref{stokes1}) reads
\begin{eqnarray*}
\int_{\om_M}   \frac{\partial (\x^\alpha \,f\,\exp(-\Vert\x\Vert^2/2))}{\partial x_i}\,d\x&=&
\int_{\om_M}p_{i,\alpha}\,d\mu\\
&=&\underbrace{
\int_{\partial\om_M^1}\langle e_i,\vec{n}_\x\rangle\, \x^\alpha\,f\exp(-\Vert\x\Vert^2/2)\,d\sigma}_{=0}\\
&&+\int_{\partial\om_M^2}\langle e_i,\vec{n}_\x\rangle\,\x^\alpha\, f\exp(-\Vert\x\Vert^2/2)\,d\sigma\\
&=&\frac{\exp(-M^2/2)}{M}\,\int_{\partial\om_M^2}x_i\,\x^\alpha\,f\,\,d\sigma
\end{eqnarray*}
and so
\[\left\vert\int_{\om_M} p_{i,\alpha}\,d\mu\,\right\vert
\leq\,\frac{2\,\pi^{\frac{n+1}{2}}\Vert f\Vert_1\,M^{\vert\alpha\vert+n+{\rm deg}f}}{\Gamma((n+1)/2)}\,\,\exp(-M^2/2),\]
because $\sup\{\vert f(\x)\vert:\Vert\x\Vert^2=M^2\}\leq \Vert f\Vert_1\,\sup\{\vert\x^\alpha\vert:\x\in\B_M\}\leq \Vert f\Vert_1M^{{\rm deg}f}$,
with $\Vert f\Vert_1=\sum_\alpha \vert f_\alpha\vert$. This in turn implies
\begin{equation}
\label{aux333}
\lim_{M\to\infty}\,\int_{\om_M} p_{i,\alpha}\,d\mu\,=\,0,\qquad\forall\alpha\in\N^n;\quad i=1,\ldots,n.
\end{equation}
It remains to show that
\[\int_{\om_M} p_{i,\alpha}\,d\mu\:\to\:\int_\om p_{i,\alpha}\,d\mu\quad\mbox{ as $M\to\infty$}.\]
Observe that
\begin{eqnarray*}
\left\vert p_{i,\alpha}(\x)\right\vert I_{\om_M}(\x)\,\exp(-\Vert\x\Vert^2/2)
&\leq&\left\vert p_{i,\alpha}(\x)\right\vert\,\exp(-\Vert\x\Vert^2/2),\quad\forall \x\in\R^n,\\
\mbox{and}\quad\int \left\vert p_{i,\alpha}(\x)\right\vert\,\exp(-\Vert\x\Vert^2/2)\,d\x&<&\infty.
\end{eqnarray*}
In addition, for every $\x\in\R^n$:
\[p_{i,\alpha}(\x)\,I_{\om_M}(\x)\,\exp(-\Vert\x\Vert^2/2)\:\to\: 
p_{i,\alpha}(\x)\,I_{\om}(\x)\,\exp(-\Vert\x\Vert^2/2)\quad\mbox{as $M\to\infty$.}\]
Therefore by (\ref{aux333}) and invoking the Dominated Convergence Theorem \cite{ash} one obtains the desired result
\begin{eqnarray*}
0\,=\,\lim_{M\to\infty}\int_{\om_M} p_{i,\alpha}\,d\mu&=&
\lim_{M\to\infty}\int p_{i,\alpha}(\x)\,I_{\om_M}(\x)\,\exp(-\Vert\x\Vert^2/2)\,d\x\\
&=&\int p_{i,\alpha}(\x)\,I_{\om}(\x)\,\exp(-\Vert\x\Vert^2/2)\,d\x\,=\,
\int_\om p_{i,\alpha}\,d\mu.
\end{eqnarray*}
\end{proof}

So with $p_{i,\alpha}\in\R[\x]$ as in (\ref{palpha-gaussian}) (or (\ref{palpha-exp}))
and $r(d):=2d-1$ (or $r(d):=2d$), we can now consider the new semidefinite program:
\begin{equation}
\label{scheme-3-primal}
\begin{array}{rl}
\overline{\omega}_d\,=\,\displaystyle\sup_{\u,\v\in\R[\x]_{2d}^*}&\{\,u_0:\:\mbox{ s.t. }
u_\alpha+v_\alpha\,=\,y_\alpha,\quad\forall\,\alpha\in\N^n_{2d};\\
&\M_d(\u),\,\M_d(\v)\succeq0;\:\M_{d-d_j}(g_j\,\u)\succeq0,\quad j=1,\ldots,m,\\
&\\
&L_\u(p_{i,\alpha})\,=\,0,\qquad \forall\,\alpha\,\in\N^n_{r(d)};\:i=1,\ldots,n\,\},
\end{array}
\end{equation}
which is a relaxation of the infinite-dimensional linear program (\ref{abstract-primal-new}) (with $f=1$). The dual of (\ref{scheme-3-primal})
is the semidefinite program:
\begin{equation}
\label{scheme-3-dual}
\begin{array}{rl}
\overline{\omega}^*_d\,=\,\displaystyle\inf_{p\in\R[\x]_{2d},\boldsymbol{\theta}}&\,\{\,\displaystyle \int  p\,d\mu:\: \mbox{s.t.}
\quad\boldsymbol{\theta}\in\R^{nm(d)},\:p\in\Sigma[\x]_d;\\
&p+\displaystyle\sum_{i=1}^n\sum_{\alpha\in\N^n_{m(d)}}
\theta_{i\alpha}\,p_{i,\alpha}-1\in Q_{2d}(\om),
\end{array}
\end{equation}
where $m(d)={n+r(d)\choose n}$ and $\boldsymbol{\theta}=(\theta_{i,\alpha})\in\R^{nm(d)}$ is the vector of dual variables associated with the equality constraints $L_\u(p_{i,\alpha})=0$, $\alpha\in\N^n_{r(d)}$, $i=1,\ldots,n$.

Again, if $\mu$ is the exponential measure (\ref{mu-exp}) then $\om$ is as in (\ref{newsetk})
and so in (\ref{scheme-3-primal}) the constraints $\M_{d-d_j}(g_j\,\u)\succeq0$, $j=1,\ldots,m$, become
$\M_{d-d_j}(g_j\,\u)\succeq0$, $j=1,\ldots,m+n$; see Remark \ref{support}.
\begin{thm}
\label{th3}
Let $\om  \subset\R^n$ be as in (\ref{setk}) (or as in (\ref{newsetk}) if $\mu$ is the exponential measure) and assume that
both $\om$  and $\supmu\setminus\om  $ have nonempty interior. Then:

(a) Both problems (\ref{scheme-3-primal}) and its dual
have an optimal solution and $\overline{\omega}_d=\overline{\omega}^*_d$ for every $d\geq\max_jd_j$.

(b) The sequence $(\overline{\omega}_d)$, $d\in\N$, is monotone non-increasing and 
$\overline{\omega}_d\to\mu(\om)$ as $d\to\infty$.
\end{thm}
\begin{proof}
(a) follows from an almost verbatim copy of the proof of Theorem \ref{th1}. Concerning
(b) it suffices to observe that $\ov_d\geq\overline{\omega}_d\geq \mu(\om)$ for every $d\geq d_0$.
Therefore the convergence $\overline{\omega}_d\to\mu(\om)$ is a consequence of the convergence $\ov_d\to\mu(\om)$.
Finally the convergence is monotone since by construction the sequence
$(\overline{\omega}_d)$, $d\in\N$, is monotone non-increasing. 
\end{proof}

Next, exactly as we did in \S \ref{lower} we can provide an associated sequence
of lower bounds $(\underline{\omega}_d)$, $d\in\N$, by considering the complement 
$\om^c:=\supmu\setminus\om$ in $\supmu$ 
(with $\om:=\om\cap\supmu$) and its decomposition (\ref{complement}). Then the analogue of Corollary
\ref{cor-1} reads:
\begin{cor}
\label{cor-2}
Let $\om  \subset\R^n$ be as in (\ref{setk}). Assume that
both $\om$  and $\om^c$ have nonempty interior and let $\om^c$ be as in (\ref{complement}).
For each $\ell=1,\ldots,s$, let $\overline{\omega}^c_{\ell d}$ be the optimal value of the semidefinite 
program (\ref{scheme-3-primal}) with $g_{\ell j}$ in lieu of $g_j$ and $m_\ell$ in lieu of $m$. 
For every $d\geq d_0$ let :
\begin{equation}
\label{cor-2-1}
\underline{\omega}_d\,:=\,\mu(\R^n)-\left(\sum_{\ell=1}^s\overline{\omega}^c_{\ell d}\,\right).
\end{equation}
Then the sequence $\underline{\omega}_d$, $d\in\N$, is monotone non-decreasing
with $\mu(\om)\geq\underline{\omega}_d$ for all $d\geq d_0$ and 
$\underline{\omega}_d\to\mu(\om)$ as $d\to\infty$.
\end{cor}
The proof being similar to that of Corollary \ref{cor-1} is omitted.
So again but now by Theorem \ref{th3} and  Corollary \ref{cor-2} one has :
\begin{equation}
\label{conclusion-11}
\underline{\omega}_{d}\,\leq\, \mu(\om)\,\leq\, \overline{\omega}_d,\quad \forall d\geq d_0\quad\mbox{and}\quad  
\lim_{d\to\infty}\,\underline{\omega}_{d}\,=\,\mu(\om)\,=\,\lim_{d\to\infty}\,\overline{\omega}_{d}.
\end{equation}

\section{Some numerical experiments}

In this section we provide some examples to illustrate the methodology developed in \S \ref{method}.
The examples are all 2D-examples as 3D-examples would have required some help
for implementation in GloptiPoly. Also and importantly,
for simplicity of implementation of  the semidefinite programs (\ref{scheme-3-primal})
we have chosen to express the moment and localizing matrices $\M_d(\z)$ and $\M_{d-d_j}(g_j\,\z)$
in the standard monomial basis $(\x^\alpha)$, $\alpha\in\N^n$, which is the worst choice from a numerical point
of view. Indeed, it is well known that the Hankel-type moment and localizing matrices matrix become rapidly
ill-conditioned as their size increases. A much better choice would the basis of Hermite polynomials that are
orthogonal with respect to the Gaussian measure on $\R^n$ (and similarly for the exponential measure). However this would require a non trivial implementation work beyond the scope of the present paper. Nonetheless
the 2D-examples provided below already give some indications on the potential of the method since
in many cases (but not all) good approximations $\underline{\omega}_d\leq\mu(\om)\leq\overline{\omega}_d$ are obtained with relatively small $d$ (say $d=8,9$ or $10$).

\subsection{Numerical experiments for the Gaussian measure}

We have considered  the Gaussian measures 
$d\mu=\exp(\Vert\x\Vert^2/\sigma^2)$ for the three values $\sigma^2=0.5,0.8,1$.
In a first set of experiments we have tested the methodology on some examples of compact sets $\om\subset\R^2$
where we could compute a very good approximate value of  $\mu(\om)$ by other methods. 
For each example we have compared the hierarchy of semidefinite programs (\ref{scheme-1-primal}) with
the hierarchy (\ref{scheme-3-primal}). Then we have tested the two hierarchies on
some simple example of non-compact sets $\om\subset\R^2$ for which the value
$\mu(\om)$ can be obtained exactly.

\subsection*{On compact sets $\om\subset\R^2$}
\begin{ex}
\label{ex1}
{\rm Consider non-centered Euclidean balls $\om:=\{\x: \Vert \x-\u\Vert^2\leq 1\}$ with
\[\u= (0,0),\quad (0.1,0.1),\quad (0.5,0.5),\quad (0.1,0.5),\quad (1,0),\]
that is,
\[\om=\{\,\x\::\quad g(\x)\,\leq\,1\,\},\quad\mbox{with}\quad \x\mapsto g(\x)\,=\,(x_1-u_1)^2+(x_2-u_2)^2.\]
We first show how the formulation (\ref{scheme-3-primal}) with $f=1$ 
drastically improves convergence when comparing with the initial formulation 
(\ref{scheme-1-primal}) with $f=1$ and $f=1-g$ (which vanishes on $\partial\om$)
for the example where $\sigma=0.5$ and $\u=(0.5,0.5)$.
Indeed in Figure \ref{figure-1} one may see that 
the convergence is very slow for the hierarchy (\ref{scheme-1-primal}) with $f=1$ (left) whereas it is very fast
for the hierarchy (\ref{scheme-3-primal}) (right). In Figure \ref{figure-2} are plotted the
values $\overline{\omega}_d$ and the values
$u_0^d$ obtained in (\ref{scheme-1-primal}) with $f=1-g$.
One can see that the convergence $u^d_0\to\mu(\om)$
of the hierarchy (\ref{scheme-1-primal}) with $f=1-g$ is not monotone and not as good as (\ref{scheme-3-primal}).
(The yellow horizontal line is the value of $\mu(\om)$, exact up to 6 digits.)
Iteration $1$ was run with $d=3$ (i.e. with moments up to order $6$) and $d+1,\ldots,9$ (i.e., with moments up of order $20$).
With $d=7$ (i.e. with moments up to order $14$) the optimal value $s_7=0.573324$ 
obtained in (\ref{scheme-3-primal}) is already very close to the value $0.573132$ (exact up to $6$ digits).

\begin{figure}[ht]
\begin{center}
\includegraphics[width=0.45\textwidth]{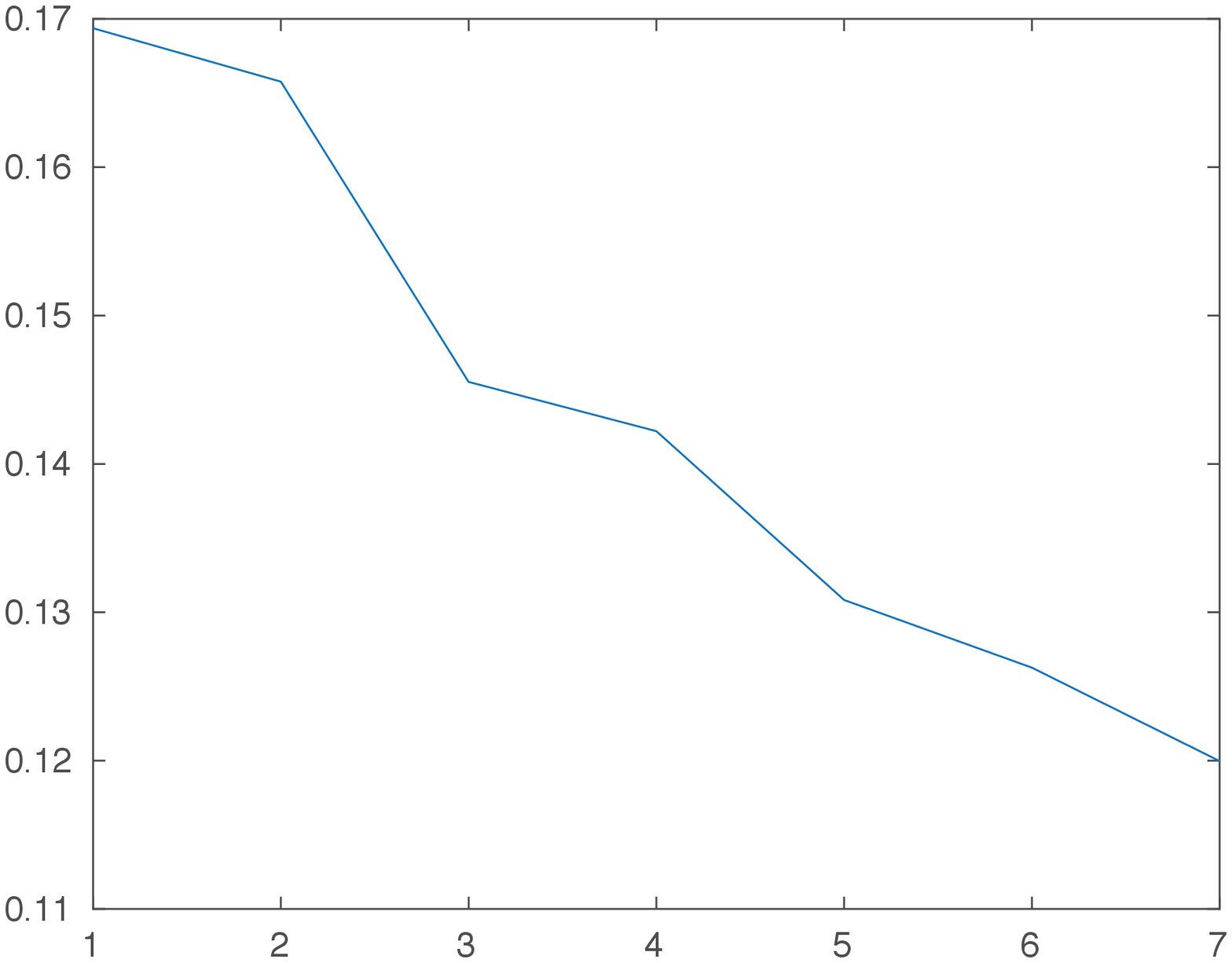}
\includegraphics[width=0.45\textwidth]{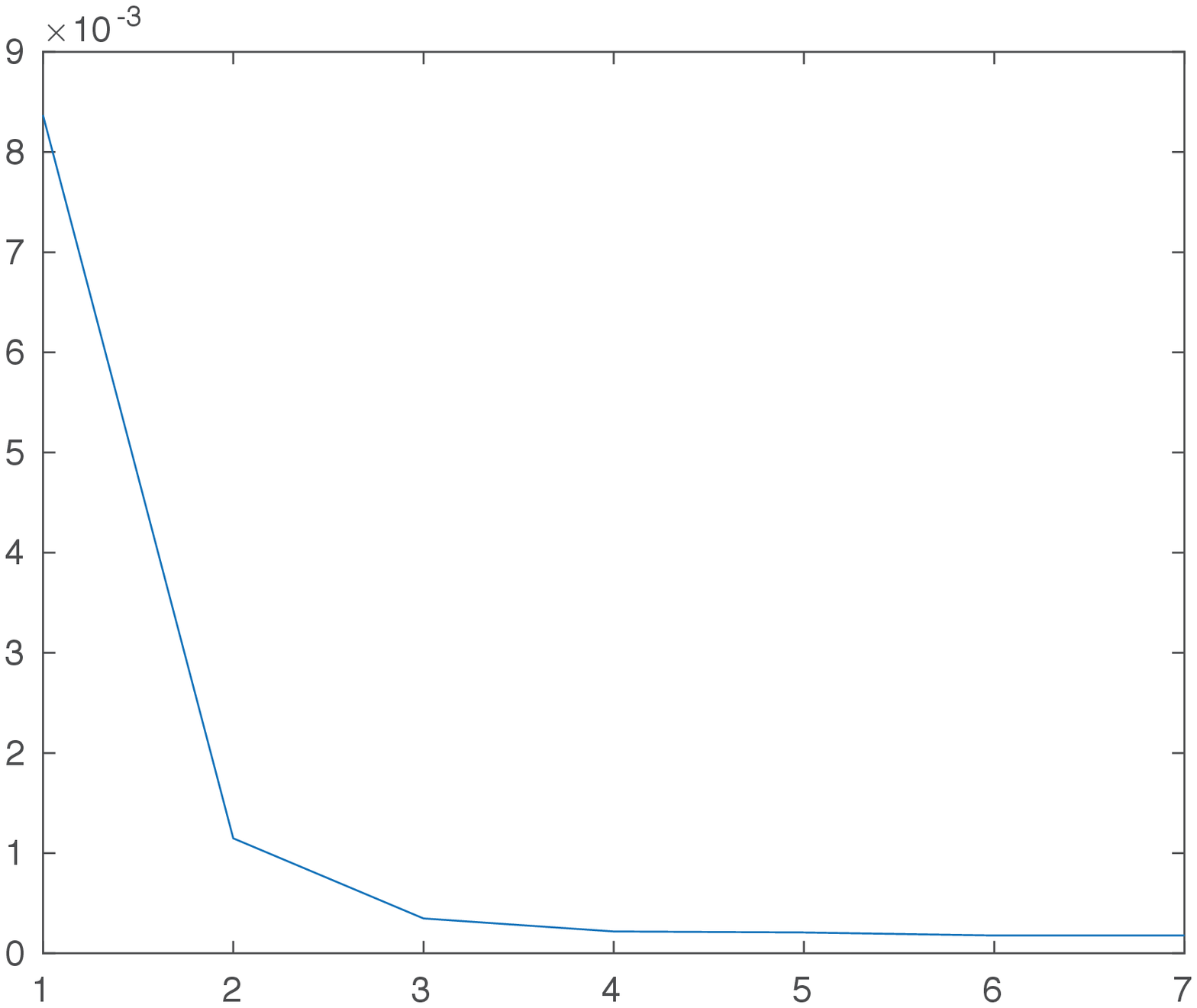}
\caption{Example \ref{ex1}: $\sigma=0.5$, $\u=(0.5,0.5)$; 
Comparing $\overline{o}_d-\mu(\om)$ with $f=1$ (left) and 
$\overline{\omega}_d-\mu(\om)$  (right). \label{figure-1}}
\end{center}
\end{figure}

\begin{figure}[ht]
\begin{center}
\includegraphics[width=0.6\textwidth]{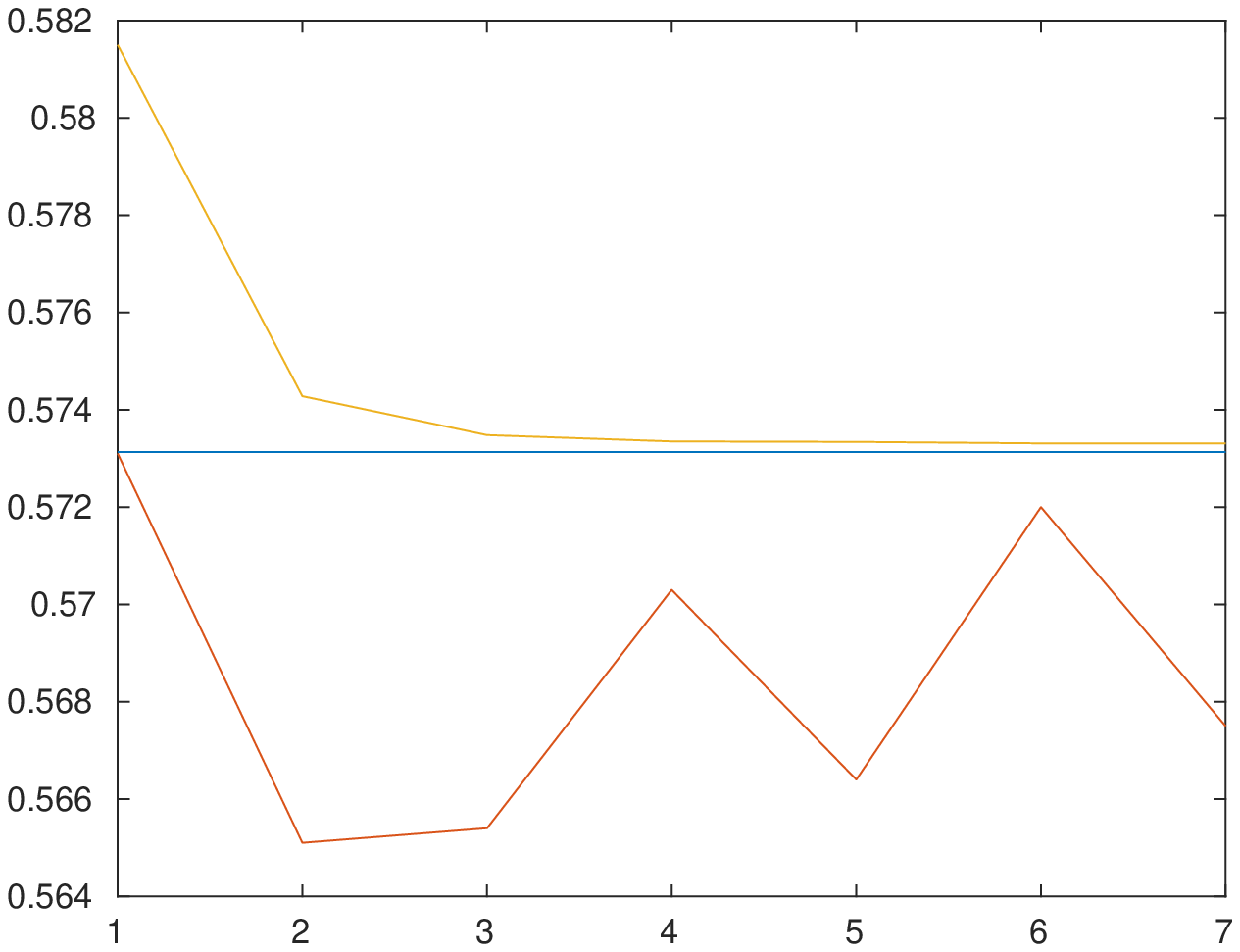}
\caption{Example \ref{ex1}: $\sigma=0.5$, $\u=(0.5,0.5)$; Comparing (\ref{scheme-3-primal}) and (\ref{scheme-1-primal}) with
$f=1-g$ \label{figure-2}; the horizontal line is the exact value.}
\end{center}
\end{figure}
\begin{figure}[ht]
\begin{center}
{\includegraphics[width=0.45\textwidth]{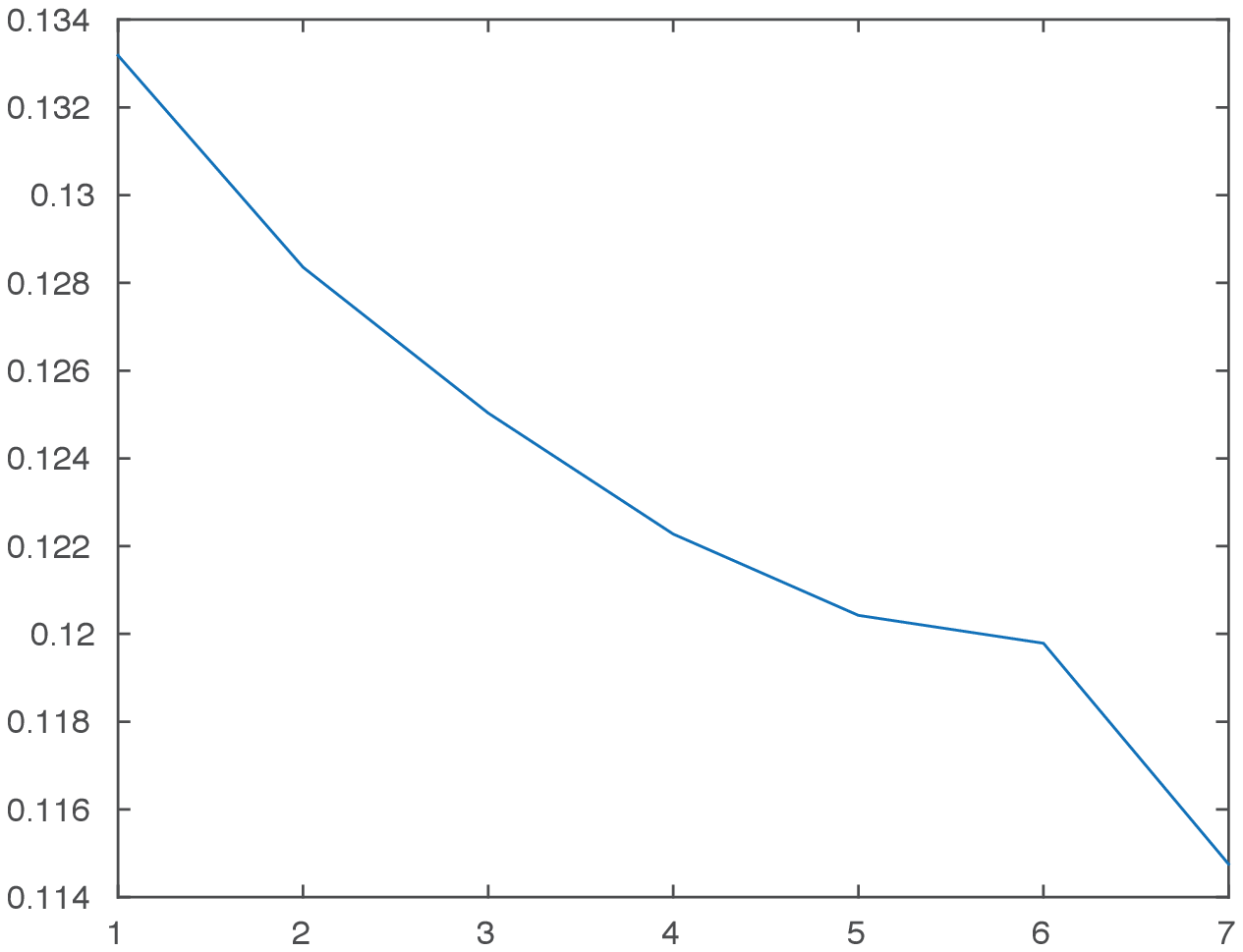}
\includegraphics[width=0.45\textwidth]{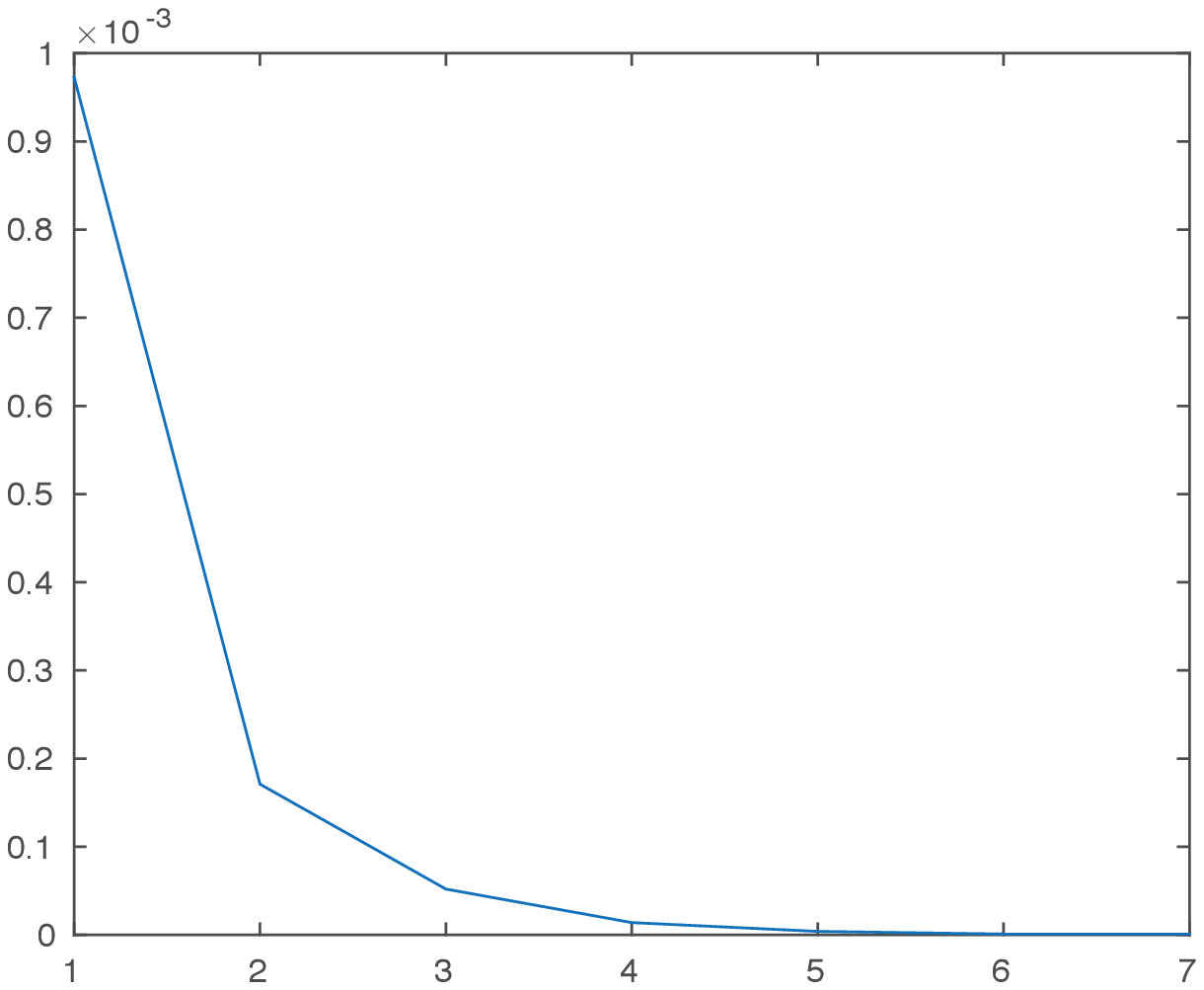}}
\caption{Example \ref{ex1}: $\overline{o}_d-\mu(\om)$ (scheme (\ref{scheme-1-primal}) with $f=1$, left) and  
$\overline{\omega}_d-\mu(\om)$ (scheme (\ref{scheme-3-primal}) right) \label{figure-3}}
\end{center}
\end{figure}
\begin{figure}[ht]
\begin{center}
{\includegraphics[width=0.45\textwidth]{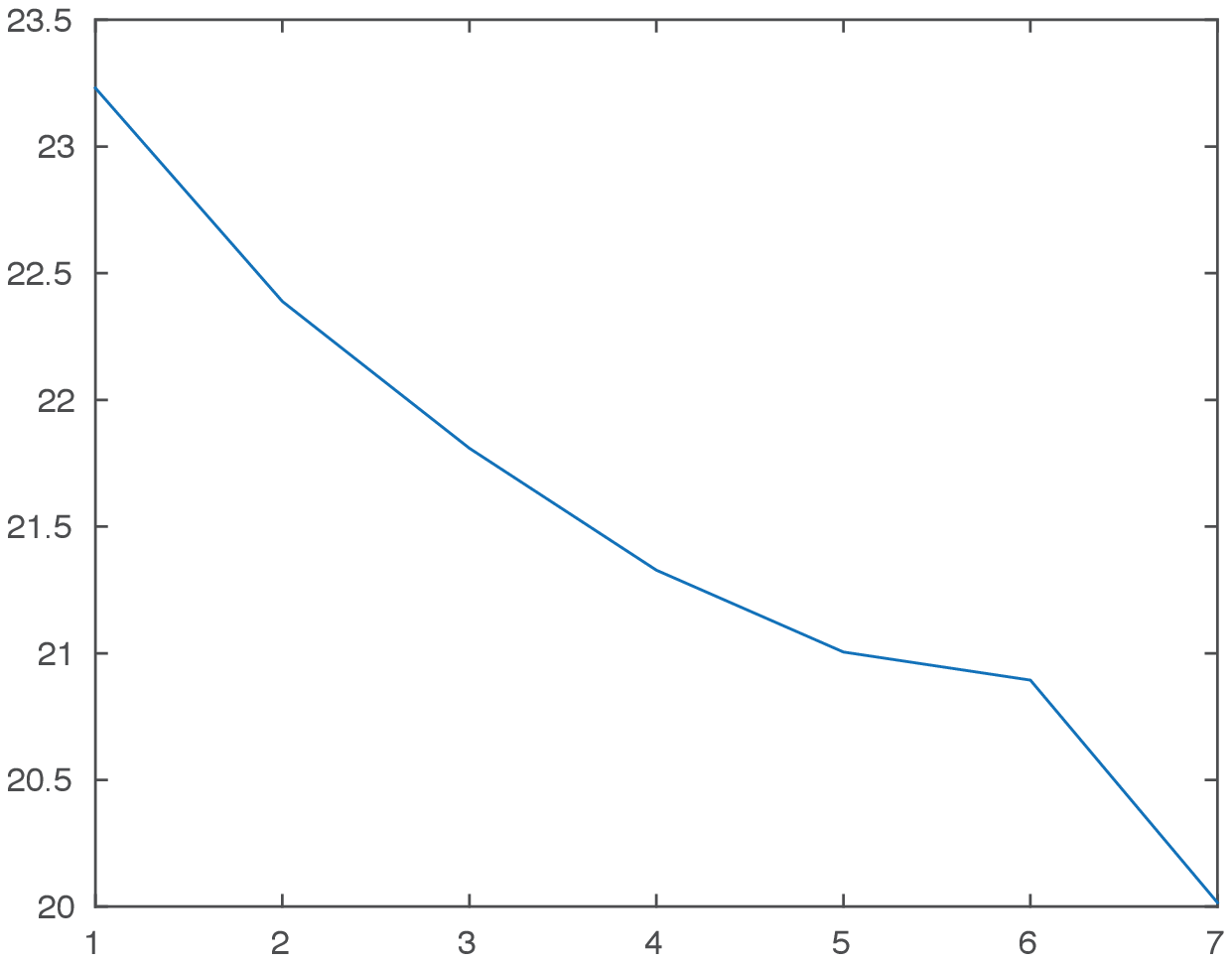}
\includegraphics[width=0.45\textwidth]{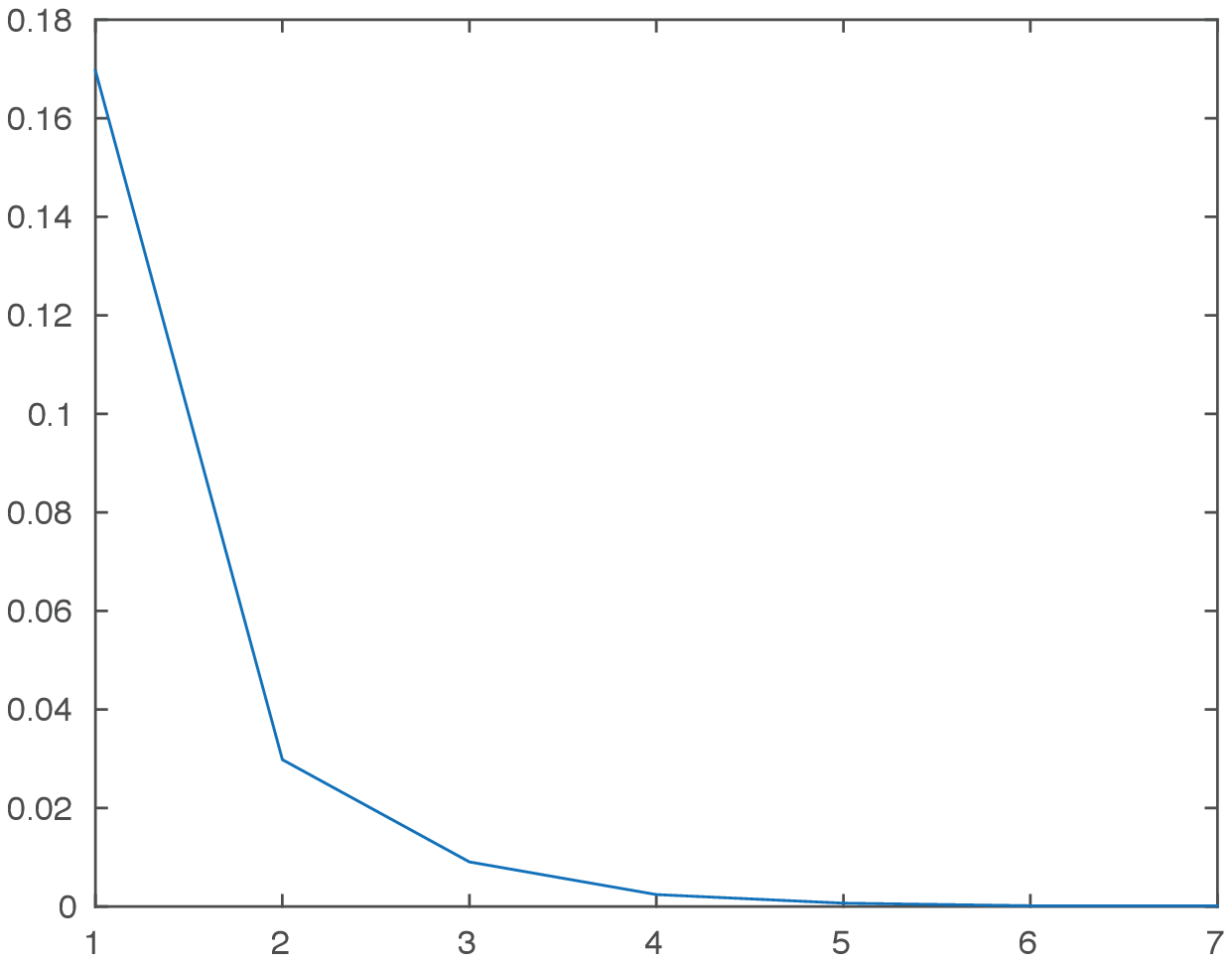}}
\caption{Comparing the relative errors $100(\ov_d-\mu(\om))/\mu(\om)$ (scheme (\ref{scheme-1-primal}) with $f=1$, left) with 
$100(\overline{\omega}_d-\mu(\om))/\mu(\om)$ (scheme (\ref{scheme-3-primal}), right) \label{figure-4}}
\end{center}
\end{figure}
For some other value of the parameters $\sigma$ and $\u$,
Figure \ref{figure-3} displays the difference $\ov_d-\mu(\om)$ for the hierarchy
(\ref{scheme-1-primal}) with $f=1$ and $\overline{\omega}_d-\mu(\om)$ for the hierarchy (\ref{scheme-3-primal}),
whereas Figure \ref{figure-4} displays the respective relative errors.

Results obtained for several combinations of the parameters  $\sigma$ and $\u$
are displayed in Tables \ref{table11-sigma=0.5},
\ref{table11-sigma=0.8} and 
\ref{table11-sigma=0.999}.
\begin{table}[!h]
\begin{center}
\begin{tabular}{||l|r|r|r|r|r||}
\hline
	 & $\u=(0,0)$	&$\u=(0.1,0.1)$&$\u=(0.1,0.5)$& $\u=(0.5,0.5)$ &$\u=(1,0)$\\
\hline
$\overline{\omega}_7$    &0.782888&0.781978 &0.783345	&  0.753521    	&  0.441636	   \\
$\underline{\omega}_7$   & 0.782842	& 0.781932 &0.783238	&   0.753492  	& 0.441631	   \\
$\epsilon_7$    & $0.005\%$	& $0.006\%$ &$0.01\%$	&   $0.003\%$   	& $0.001\%$\\
\hline
\end{tabular}
\end{center}
\caption{Values of $\overline{\omega}_7$, $\underline{\omega}_7$
and $\epsilon_7:=100(\overline{\omega}_7-\underline{\omega}_7)/\underline{\omega}_7$
with $\sigma=.5$ \label{table11-sigma=0.5}}
\end{table}
\begin{table}[!h]
\begin{center}
\begin{tabular}{||l|r|r|r|r|r||}
\hline
	 & $\u=(0,0)$	&$\u=(0.1,0.1)$&$\u=(0.1,0.5)$& $\u=(0.5,0.5)$ &$\u=(1,0)$\\
\hline
$\overline{\omega}_7$    &1.906281&1.900152 &1.920172	&  1.770550  	&  1.183305	   \\
$\underline{\omega}_7$   & 1.905856	& 1.899708 &1.919838	&   1.770276   	& 1.182967	   \\
$\epsilon_7$    & $0.02\%$	& $0.02\%$ &$0.017\%$	&   $0.015\%$   	& $0.028\%$\\
\hline
\end{tabular}
\end{center}
\caption{Values of $\overline{\omega}_7$, $\underline{\omega}_7$
and $\epsilon_7:=100(\overline{\omega}_7-\underline{\omega}_7)/\underline{\omega}_7$
with $\sigma=.8$ \label{table11-sigma=0.8}}
\end{table}
\begin{table}[!h]
\begin{center}
\begin{tabular}{||l|r|r|r|r|r||}
\hline
	 & $\u=(0,0)$	&$\u=(0.1,0.1)$&$\u=(0.1,0.5)$& $\u=(0.5,0.5)$ &$\u=(1,0)$\\
\hline
$\overline{\omega}_7$    &2.811312&2.802918 &2.843411	&  2.624901   	&  2.947715	   \\
$\underline{\omega}_7$   &2.806004 	& 2.797471 &2.839425	& 2.621340    	& 2.946099	   \\
$\epsilon_7$    & $0.19\%$	& $0.19\%$ &$0.14\%$	&   $0.13\%$   	& $0.05\%$\\
\hline
\end{tabular}
\end{center}
\caption{Values of $\overline{\omega}_7$, $\underline{\omega}_7$
and $\epsilon_7:=100(\overline{\omega}_7-\underline{\omega}_7)/\underline{\omega}_7$
with $\sigma=.999$ \label{table11-sigma=0.999}}
\end{table}
}
\end{ex}
\begin{ex}
\label{ex2}
{\rm Consider a set $\om:=\{\x\in\R^2: g(\x)\,\leq\,1\}$ with
$\x\mapsto g(\x):=(\x-\u)^T\A\,(\x-\u)$ with:
\[\A\,:=\,
\left[\begin{array}{cc}0.4& 0.1\\ 0.1 &-0.4\end{array}\right]\,
\left[\begin{array}{cc}4& 0\\ 0 &8\end{array}\right]\,
\left[\begin{array}{cc}0.4& 0.1\\ 0.1 &-0.4\end{array}\right]\]
and $\u=(0.1,0.5)$, $\u=(0.5,0.1)$. The results displayed in Table \ref{table-xax1}
show that good approximations can be obtained with relatively small $d=9$ when $\sigma\leq 1$.
\begin{table}[!h]
\begin{center}
\begin{tabular}{||l|c|c|r||}
\hline
&\multicolumn{3}{c||}{$\u=(0.1,0.5)$}\\
\hline
$\sigma$	 & $\overline{\omega}_{9}$	&$\underline{\omega}_{9}$&$(\overline{\omega}_{9}-\underline{\omega}_{9})/\underline{\omega}_{9})$\\
\hline
1   & 1.728640	& 1.657139& 4.3\%  \\
0.8   & 1.305168	& 1.298370& 0.5\%  \\
0.5 &  0.648429 & 0.648415 & 0.002\%\\
\hline
&\multicolumn{3}{c||}{$\u=(0.5,0.1)$}\\
\hline
$\sigma$	 & $\overline{\omega}_{9}$	&$\underline{\omega}_{9}$&$(\overline{\omega}_{9}-\underline{\omega}_{9})/\underline{\omega}_{9})$\\
\hline
1   &  1.798363	& 1.730477& 3.9\%  \\
0.8   & 1.401565& 1.394053 & 0.53\%  \\
0.5 &  0.724022 & 0.724005 & 0.002\%\\
\hline
\end{tabular}
\end{center}
\caption{Example \ref{ex2}: $\om=\{\x:(\x-\u)^T\A(\x-\u)\leq 1\}$; compact case.\label{table-xax1}}
\end{table}

}
\end{ex}
\subsection*{Non-compact sets $\om\subset\R^2$}

We next provide some examples of non-compact sets 
$\om\subset\R^2$ and show that  the hierarchy (\ref{scheme-3-primal}) can still provide good results.

\begin{ex}
\label{ex3}
{\rm {\bf Complement of Euclidean balls:} We have first considered evaluating $\mu(\om^c)$ for the complement $\om^c:=\{\x: \Vert\x-\u\Vert^2\geq 1\}$ 
of the sets $\om$ considered in Example \ref{ex1}. 
As the lower bound $\underline{\omega}_d$ on $\mu(\om)$ in Example \ref{ex1} was computed 
via an upper bound on $\mu(\om^c)$, the results in Table \ref{table11-sigma=0.5},
Table \ref{table11-sigma=0.8} and Table \ref{table11-sigma=0.999} 
clearly indicate that
good upper and lower bounds are also obtained for $\om^c$.

}\end{ex}
\begin{ex}
\label{ex4}
{\rm {\bf Non-convex quadratics:} 
Let $\u=(0.1,0.5)$, $\u=(0.5,1)$, 
\[\A\,:=\,
\left[\begin{array}{cc}0.4& 0.1\\ 0.1 &-0.4\end{array}\right]\,
\left[\begin{array}{cc}4& 0\\ 0 &-8\end{array}\right]\,
\left[\begin{array}{cc}0.4& 0.1\\ 0.1 &-0.4\end{array}\right],\]
and $\x\mapsto g(\x)=(\x-\u)^T\A(\x-\u)$ so that the set $\om:=\{\x: g(\x)\leq 1\}$ is non-compact. The results displayed
in Table \ref{table-xax2} show that good approximations can be obtained with relatively small $d=9$ when $\sigma\leq 1$.
\begin{table}[!h]
\begin{center}
\begin{tabular}{||l|c|c|r||}
\hline
&\multicolumn{3}{c||}{$\u=(0.5,0.1)$}\\
\hline
$\sigma$	 & $\overline{\omega}_{9}$	&$\underline{\omega}_{9}$&$(\overline{\omega}_{9}-\underline{\omega}_{9})/\underline{\omega}_{9})$\\
\hline
1   &  2.829605	& 2.824718& 0.17\%  \\
0.8   & 1.876731& 1.876609 & 0.006\%  \\
\hline
&\multicolumn{3}{c||}{$\u=(0.1,0.5)$}\\
\hline
$\sigma$	 & $\overline{\omega}_{9}$	&$\underline{\omega}_{9}$&$(\overline{\omega}_{9}-\underline{\omega}_{9})/\underline{\omega}_{9})$\\
\hline
1   &  2.989832	& 2.986599& 0.10\%  \\
0.8   & 1.969188& 1.969103 & 0.004\%  \\
\hline
\end{tabular}
\end{center}
\caption{$\om=\{\x:(\x-\u)^T\A(\x-\u)\leq 1\}$; non-compact case.  \label{table-xax2}}
\end{table}
}\end{ex}
\begin{ex}
\label{ex-5}
{\rm {\bf Half-spaces}: In this example we consider the half-space $\om:=\{\x\in\R^2:x_1+2x_2\geq 1\}$,
and compute the upper and lower bounds $\overline{\omega}_d$ and $\underline{\omega}_d$ for $d=8$.
the results displayed in Table \ref{table-halfspace} show that 
good approximations can be obtained with relatively small $d=8$ when $\sigma\leq 1$.
\begin{table}[!h]
\begin{center}
\begin{tabular}{||l|c|c|c||}
\hline
$\sigma$	 & $\overline{\omega}_8$	&$\underline{\omega}_8$&$(\overline{\omega}_8-\underline{\omega}_8)/\underline{\omega}_8)$\\
\hline
1.0   & 0.828105	& 0.827800& 0.03\%  \\
\hline
0.8   & 0.4314786	& 0.431473& 0.001\%  \\
\hline
0.5 & 0.080858 & 0.0808578 & 0.0003\%\\
\hline
\end{tabular}
\end{center}
\caption{$\om:=\{\x\in\R^2:x_1+2x_2\geq 1\}$ \label{table-halfspace}}
\end{table}
}\end{ex}
\begin{ex}
\label{ex-3}
{\rm In this example we consider the cone $\om:=\{\x\in\R^2:x_1+2x_2\leq -0.5;\:x_1\geq-0.8\}$,
with apex $\x=(-0.8,0.15)$,
and compute the upper and lower bounds $\overline{\omega}_d$ and $\underline{\omega}_d$ for $d=10$.

\begin{table}[!h]
\begin{center}
\begin{tabular}{||l|c|c|c||}
\hline
$\sigma$	 & $\overline{\omega}_{10}$	&$\underline{\omega}_{10}$&$(\overline{\omega}_{10}-\underline{\omega}_{10})/\underline{\omega}_{10})$\\
\hline
0.5   & 0.208606	& 0.19410& 7\%  \\
\hline
0.4   & 0.108321	& 0.10768& 0.6\%  \\
\hline
0.3 & 0.041300 & 0.0411906 & 0.26\%\\
\hline
\end{tabular}
\end{center}
\caption{$\om:=\{\x\in\R^2:x_1+2x_2\leq -0.5;\:x_1\geq-0.8\}$ \label{tablecone}}
\end{table}
Table \ref{tablecone} shows that when $\sigma$ is relatively small then a good approximation
can be obtained by using no more than $2d=20$ moments. 
On the other hand for the cone $\om:=\{\x:\x\geq0\}$, even with $\sigma=0.3$ the approximation with 
$d=10$ is rather rough as we have $\overline{\omega}_{10}=0.11172\geq \mu(\om)\approx0.070685$ and lower bound being negative is not informative. It seems that this is because the origin (the apex of the cone $\om$) is in a region where the density 
$\exp(-h)$ is maximum. In this case higher order moments are needed for a better approximation. This in turn requires
some care in solving the associated semidefinite programs (\ref{scheme-3-primal}). Indeed,
when expressed in the standard monomial basis $(\x^\alpha)$, $\alpha\in\N$,
the moment and localizing matrices in (\ref{scheme-3-primal}) become
ill-conditioned. An interesting issue of further investigation and beyond the scope of this paper is to
consider alternative bases, e.g. the basis of Hermite polynomials which are orthogonal with respect to the Gaussian measure.
}\end{ex}

\subsection{Numerical experiments for the exponential measure}

In this second set of experiments we have chosen the exponential measure with density
$\x\mapsto \exp(-\lambda\sum_i \,x_i)$ on the positive orthant $\R^n_+=\{\x: \x\geq0\}$, for some real scalar $\lambda>0$.

\subsection*{Compact sets}
\begin{ex} {\bf Ellipsoids:}
\label{ex21}
{\rm As in Example \ref{ex2}, consider the set $\om:=\{\x\in\R^2: g(\x)\,\leq\,1\}$ with
$\x\mapsto g(\x):=(\x-\u)^T\A\,(\x-\u)$ with:
\[\A\,:=\,
\left[\begin{array}{cc}0.4& 0.1\\ 0.1 &-0.4\end{array}\right]\,
\left[\begin{array}{cc}4& 0\\ 0 &8\end{array}\right]\,
\left[\begin{array}{cc}0.4& 0.1\\ 0.1 &-0.4\end{array}\right]\]
and $\u=(0.1,0.5)$, $\u=(0.5,0.1)$.
\begin{table}[!h]
\begin{center}
\begin{tabular}{||l|c|c|c||}
\hline
&\multicolumn{3}{c||}{$\u=(0.1,0.5)$}\\
\hline
$\lambda$	 & $\overline{\omega}_{9}$	&$\underline{\omega}_{9}$&$(\overline{\omega}_{9}-\underline{\omega}_{9})/\underline{\omega}_{9})$\\
\hline
4   & 0.061906	&0.060311 & 2.64\%  \\
4.5   & 0.049080	&0.048498 & 1.2\%  \\
5   & 0.039844	& 0.039619& 0.56\%  \\
\hline
&\multicolumn{3}{c||}{$\u=(0.5,0.1)$}\\
\hline
$\sigma$	 & $\overline{\omega}_{9}$	&$\underline{\omega}_{9}$&$(\overline{\omega}_{9}-\underline{\omega}_{9})/\underline{\omega}_{9})$\\
\hline
4&  0.061394 & 0.059065& 3.94\%\\
4.5   &  0.048757	& 0.047922 & 1.74\%  \\
5   & 0.039622&0.039283  & 0.86\%  \\

\hline
\end{tabular}
\end{center}
\caption{$\om=\{\x:(\x-\u)^T\A(\x-\u)\leq 1\}$; compact case.\label{table-xax-exp}}
\end{table}
}
\end{ex}
\begin{ex} {\bf Simplex:}
\label{ex22}
{\rm Let  $\om:=\{\x\in\R^2_+: 3x_1+x_2\leq1\}$.
In this case it is easy to compute the exact value $\mu(\om)$ by Laplace transform techniques, which gives
\begin{equation}
\label{exact-exp}
\mu(\om)=\frac{1}{\lambda^2}\,\left(1+\frac{\exp(-p)}{2}-\frac{3\exp(-p/3)}{2}\right).\end{equation}
Table \ref{table-simplex-exp} displays results obtained with $d=8$.
\begin{table}[!h]
\begin{center}
\begin{tabular}{||l|c|c|c|c|c||}
\hline
$\lambda$	 & $\mu(\om)$ & $\overline{\omega}_{8}$	&$\underline{\omega}_{8}$&
$(\overline{\omega}_{8}-\underline{\omega}_{8})/\underline{\omega}_{8})$&
$(\overline{\omega}_{8}-\mu(\om))/\mu(\om)$\\
\hline
5   & 0.028802& 0.029771	& 0.028086& 6\%  &3.36\%\\
6   &0.022173 &0.022605	&0.021790 & 3.74\%  &1.95\%\\
\hline

\end{tabular}
\end{center}
\caption{Ex. \ref{ex22}: $\om=\{\x\geq0:3x_1+x_2\leq 1\}$; compact case.\label{table-simplex-exp}}
\end{table}
}
\end{ex}

\subsection*{Non-compact sets}
\begin{ex} {\bf Half-space:}
\label{ex23}
{\rm Let $\om$ be the non-compact set $\{\x\in\R^2_+: 3x_1+x_2\geq1\}$.
From (\ref{exact-exp}) we have
\[\mu(\om)=\mu(\R^2)-\mu(\om^c)
\,=\,\frac{1}{\lambda^2}\,\left(\frac{3\exp(-p/3)}{2}-\frac{\exp(-p)}{2}\right).\]
Results for $d=8$ are displayed in Table \ref{table-halfspace-exp}.
For $d=9$ some numerical instability occured.

\begin{table}[!h]
\begin{center}
\begin{tabular}{||l|c|c|c|c|c||}
\hline
$\lambda$	 & $\mu(\om)$ & $\overline{\omega}_{8}$	&$\underline{\omega}_{8}$&
$(\overline{\omega}_{8}-\underline{\omega}_{8})/\underline{\omega}_{8})$&
$(\overline{\omega}_{8}-\mu(\om))/\mu(\om)$\\
\hline
5   & 0.011197& 0.011913&0.010228& 16.4\%  &6.4\%\\
6   &0.005604 &0.005986	& 0.005171& 15.7\%  &6.8\%\\
\hline
\end{tabular}
\end{center}
\caption{Ex \ref{ex23}: $\om=\{\x\geq0;\:3x_1+x_2\geq 1\}$\label{table-halfspace-exp}}
\end{table}
}
\end{ex}
\begin{ex} {\bf Hyperbola:}
\label{ex24}
{\rm Let $\om:=\{\x\in\R^2_+: x_1x_2\leq0.1\}$.
Results for $d=8$ are displayed in Table \ref{table-halfspace-exp}.
As in Example \ref{ex23}, some numerical instability occured for $d=9$.
\begin{table}[!h]
\begin{center}
\begin{tabular}{||l|c|c|c||}
\hline
\multicolumn{4}{||c||}{$\om=\{\x\geq0;\,x_1x_2\leq 0.1\}$}\\
\hline
\hline
$\lambda$	 & $\overline{\omega}_{8}$	&$\underline{\omega}_{8}$&
$(\overline{\omega}_{8}-\underline{\omega}_{8})/\underline{\omega}_{8})$\\
\hline
5   & 0.004328& 0.003970& 9\% \\
6   &0.001701&0.001616& 5.2\\
\hline
\hline
\multicolumn{4}{||c||}{$\om=\{\x\geq0;\,x_1x_2\geq 0.1\}$}\\
\hline
\hline
$\lambda$	 & $\overline{\omega}_{8}$	&$\underline{\omega}_{8}$&
$(\overline{\omega}_{8}-\underline{\omega}_{8})/\underline{\omega}_{8})$\\
\hline
5   & 0.036015&0.035635&1\%\\
6   & 0.026166&0.026068&0.37\%\\

\hline
\end{tabular}
\end{center}
\caption{$\om=\{\x\geq0:x_1x_2\leq 0.1\}$ and $\om^c$\label{table-hyperbole}}
\end{table}
}
\end{ex}
\begin{ex} {\bf Non-convex quadratics:}
\label{ex221}
{\rm Let $\om:=\{\x\geq0: g(\x)\,\leq\,0.05\}$ with
$\x\mapsto g(\x):=(\x-\u)^T\A\,(\x-\u)$ with:
\[\A\,:=\,
\left[\begin{array}{cc}0.4& 0.1\\ 0.1 &-0.4\end{array}\right]\,
\left[\begin{array}{cc}-1& 0\\ 0 &10\end{array}\right]\,
\left[\begin{array}{cc}0.4& 0.1\\ 0.1 &-0.4\end{array}\right]\]
and $\u=(0,0)$, $\u=-0.1,0.1)$ and $\u=(0.1,-0.1)$. The results displayed in Table
\ref{table-xax3-exp} show that for the same value of $d=9$ 
the quality of the approximation can be sensitive to data. For instance
it is much worse with $\u=(-0.1,0.1)$ than with $\u=(0,0)$ or $\u=(-0.1,0.1)$.

\begin{table}[!h]
\begin{center}
\begin{tabular}{||l|c|c|c||}
\hline
&\multicolumn{3}{c||}{$\u=(0,0)$}\\
\hline
$\lambda$	 & $\overline{\omega}_{9}$	&$\underline{\omega}_{9}$&$(\overline{\omega}_{9}-\underline{\omega}_{9})/\underline{\omega}_{9})$\\
\hline
8   & 0.013444	&0.012065 & 11\%  \\
9   & 0.010961	&0.009890 & 10\%  \\
\hline
&\multicolumn{3}{c||}{$\u=(-0.1,0.1)$}\\
\hline
$\lambda$	 & $\overline{\omega}_{9}$	&$\underline{\omega}_{9}$&$(\overline{\omega}_{9}-\underline{\omega}_{9})/\underline{\omega}_{9})$\\
\hline
8&  0.014877 & 0.014288& 4.1\%\\
9   &  0.011925& 0.011512 & 3.6\%  \\
\hline
\hline
&\multicolumn{3}{c||}{$\u=(0.1,-0.1)$}\\
\hline
$\lambda$	 & $\overline{\omega}_{9}$	&$\underline{\omega}_{9}$&$(\overline{\omega}_{9}-\underline{\omega}_{9})/\underline{\omega}_{9})$\\
\hline
8&  0.009464 &0.006009& 57\%\\
\hline
8&  0.007933 &0.004929& 60\%\\
\hline
\end{tabular}
\end{center}
\caption{$\om=\{\x\geq0:(\x-\u)^T\A(\x-\u)\leq 0.05\}$;\label{table-xax3-exp}}
\end{table}
}
\end{ex}

\section{Concluding remarks}

We have provided a systematic numerical scheme to approximate 
the Gaussian or exponential measure $\mu(\om)$ of a class of semi-algebraic sets $\om\subset\R^2$.
In principle $\mu(\om)$ can be approximated as closely as desired by solving a hierarchy of semidefinite programs of increasing size $d$. Of course the size of the resulting semidefinite programs increases with $d$. Namely
the semidefinite program (\ref{scheme-3-primal}) has $2{n+2d\choose n}$ variables and both (real symmetric) moment matrices $\M_d(\u)$ and $\M_d(\v)$  are of size ${n+d\choose n}\times {n+d\choose n}$.

So as long as one is interested in $2D$ or $3D$ problems
the size is not the most serious problem. Indeed when $\mu(\om)$ is not very small compared to the mass of $\mu$
(and $\sigma<1$ for the gaussian measure
and $\lambda >5$ for the exponential measure), good results can be expected fort reasonable values of $d$ (as shown
in some illustrative 2D-examples). On the other hand if one needs high values of $d$ (e.g. when $\sigma >>1$ or $\lambda>1$) then
the precision of the SDP solvers can become a serious issue as the semidefinite programs (\ref{scheme-3-primal}) become ill-conditioned. 

In particular in such conditions the choice 
of the monomial basis $(\x^\alpha)$, $\alpha\in\N^n$, in which to express the moment and localizing matrices $\M_d(\u)$, $\M_d(g_j\,\u)$ is not appropriate. At this preliminary validation stage of the methodology it was done for simplicity and easyness of implementation 
(in order to use the software package GloptiPoly \cite{gloptipoly}).
To avoid ill-conditionning a more appropriate strategy is to use the basis of {\it Hermite polynomials}, orthogonal with respect to the Gaussian measure $\mu$. With such a choice one may expect to be able to 
solve (\ref{scheme-3-primal}) for higher values of $d$ and so obtain better upper and lower bounds.

In any case it is also worth mentioning that even if one is forced to stop the hierarchy (\ref{scheme-3-primal}) at a relatively small value
of $d$, one has still obtained a non trivial finite sequence of upper and lower bounds on $\mu(\om)$ for non-trivial (and possibly non-compact) sets $\om$.

\subsection*{Acknowledgements}
Research  funded by by the European Research Council
(ERC) under the European Union's Horizon 2020 research and innovation program
(grant agreement ERC-ADG 666981 TAMING)"


\begin{thebibliography}{las}
\bibitem{alfano}
S. Alfano. A numerical implementation of spherical objet collision probability,
J. Astro. Sci. {\bf 53}, pp.,  2005.
\bibitem{handbook}
M. Anjos, J.B. Lasserre. {\it Handbook of Semidefinite, Conic and Polynomial Optimization},
M. Anjos and J.B. Lasserre Eds., Springer-Verlag, New York, 2012.
\bibitem{ash}
R. Ash. {\it Real Analysis and Probability}, Academic Press Inc., San Diego, 1972.
\bibitem{barvinok}
A.I. Barvinok. Exponential integrals and sums over convex polyhedra, Funct. Anal. Appl. {\bf 26}, pp; 127--129, 1992.
\bibitem{chan}
F. K. Chan. {\it Spacecraft Collision Probability}, IAA. The Aerospace Press, 2008.
\bibitem{chandramouli}
R. Chandramouli, N. Ranganathan. Computing the Bivariate
Gaussian Probability Integral, IEEE Sign. Proc. Letters {\bf 6}, pp. 129--130, 1999.
\bibitem{polygons}
A.R Didonato, M.P. Jarnagin, R.K. Hageman. Computation of the integral of the bivariate notmal distribution over convex polygons, 
SIAM J. Sci. Stat. Comput. {\bf 1},  pp. 179--186.
\bibitem{federer}
H. Federer. The Gauss-Green theorem, Trans. Amer. Math. Soc. {\bf 58}, pp. 44--76, 1945.
\bibitem{sdpa}
K. Fujisawa, M. Fukuda, M. Kojima, K. Nakata. 
Numerical Evaluation of SDPA (Semidefinite Programming Algorithm), In: {\it High Performance Optimization},
Applied Optimization Volume {\bf 33}, Springer, New York, 2000, pp. 267--301.
\bibitem{genz1}
A. Genz. Numerical computation of Rectangular bivariate and trivariate Normal and t probabilities, Statistics and Computing {\bf 14}, pp. 151--160, 2004.
\bibitem{genz2}
A. Genz, J. Monahan. A Stochastic Algorithm for High Dimensional Integrals over
Unbounded Regions with Gaussian Weight, J. Comp. Appl. Math. {\bf 112}, pp. 71--81, 1999.
\bibitem{gloptipoly}
D. Henrion D., J.B. Lasserre, J. Lofberg. Gloptipoly 3: moments, optimization and semidefinite programming,
Optim. Methods \& Softwares {\bf 24},   pp. 761--779, 2009
\bibitem{sirev}
Henrion D., Lasserre J.B., Savorgnan C. Approximate volume and integration for basic semialgebraic sets,
SIAM Review {\bf 51},  pp. 722--743, 2009. 
\bibitem{kotz1}
S. Kotz, N.L. Johnson, D.W. Boyd. 
Series representation of distribution of quadratic forms in normal variables. I.
Central case, Annals Math. Stat. {\bf 38}, pp. 823--837, 1967.
\bibitem{kotz2}
S. Kotz, N.L. Johnson, D.W. Boyd. Series Representations of Distributions of Quadratic Forms in 
Normal Variables II. Non-Central Case, Annals Math. Stat. {\bf 38}, pp. 838--848, 1967.
\bibitem{lass-camb}
J.B. Lasserre. {\em An Introduction to Polynomial and Semi-Agebraic Optimization},
Cambridge University Press, Cambridge, 2015.
\bibitem{lass-tam}
J.B. Lasserre. The K-moment problem for continuous functionals,  Trans. Amer. Math. Soc. {\bf 365},  pp. 2489--2504, 2013.
\bibitem{lass-book-icp}
J.B. Lasserre. {\it Moments, {P}ositive {P}olynomials and {T}heir {A}pplications},
Imperial College Press, London, 2009.
\bibitem{patera}
R. P. Patera. General method for calculating satellite conjunction probability,
J. Guidance Control \& Dynamics {\bf 24}, pp. 2001.
\bibitem{ruben}
H. Ruben. Probability Content of Regions Under Spherical Normal Distributions, IV: The Distribution of Homogeneous and Non-Homogeneous Quadratic Functions of Normal Variables, Annals Math. Stat. {\bf 33}, pp.  542--570, 1962.
\bibitem{spatial-3}
B. Salvy. D-finiteness: Algorithms and applications. In {\it Proceedings of the 18th International Symposium on Symbolic and Algebraic Computation}, Beijing, China,
July 24--27, 2005, Manuel Kauers Editor,  ACM Press, 2005, pp. 2--3. 
\bibitem{spatial-1}
R. Serra, D. Arzelier, M. Joldes, J.B. Lasserre, A. Rondepierre, B. Salvy.
Fast and accurate computation of orbital collision probability for short-term encounters,
J. Guidance Control \& Dynamics {\bf 39}, pp. 1--13, 2016.
\bibitem{spatial-4}
D. Zeilberger. A holonomic systems approach to special functions identities. J. Comput.
Appl. Math. {\bf 32}, pp. 321--368, 1990.
\bibitem{t-integrals}
A Fortran 90 Program for Evaluation of Multivariate Normal and Multivariate t Integrals Over Convex Regions, J. Stat. Software {\bf 3}, No 4, 1999.
\bibitem{ellipsoid}
Numerical Computation of Multivariate Normal and Multivariate t Probabilities over Ellipsoidal Regions, J. Stat. Software {\bf 6}, No 8, 2001.
\bibitem{whitney}
H. Whitney. {\em Geometric Integration Theory}, Princeton University Press, Princeton, 1957.
\bibitem{wright}
S. Wright. {\em Primal-Dual Interior-Point Methods}, SIAM, Philadelphia, PA, 1997.
\end{thebibliography}
\end{document}